\documentclass[a4paper,11pt]{article}
\usepackage{amsmath, amssymb, graphicx, amsthm, fancyhdr, dsfont, algorithm, algpseudocode, hyperref, bm, bbm, fullpage, xcolor}

\usepackage{natbib}
\usepackage{changes}

\newtheorem{theo}{Theorem}[section]
\newtheorem{prop}[theo]{Proposition}
\newtheorem{lem}[theo]{Lemma}

\theoremstyle{definition}
\newtheorem{defn}[theo]{Definition}
\newtheorem{rmk}[theo]{Remark}

\allowdisplaybreaks

\newcommand{\ra}{\rightarrow}
\newcommand{\loc}{\text{loc}}
\newcommand{\Ind}{\mathbbm{1}}
\newcommand{\Zm}{\mathbb{Z}}
\newcommand{\Law}{\mathcal{L}}
\newcommand{\Rm}{\mathbb{R}}
\newcommand{\bbP}{\mathbb{P}}
\newcommand{\bbE}{\mathbb{E}}
\newcommand{\bbN}{\mathbb{N}}
\newcommand{\bbR}{\mathbb{R}}

\newcommand{\rmd}{\mathrm{d}}
\newcommand{\expE}{\mathbb{E}}
\newcommand{\Pm}{\mathbb{P}}
\renewcommand{\epsilon}{\varepsilon}

\title{Well-posedness for Wright--Fisher SPDEs with measurable drift}
\author{Jere Koskela \\
	\texttt{jere.koskela@tugraz.at} \\
	\small Institute of Statistics, \\
    \small Graz University of Technology
	\and
	Oliver Tough \\
	\texttt{oliverktough@gmail.com} \\
	\small Department of Mathematical Sciences, \\ 
    \small Durham Univeristy
}
\date{\today}

\begin{document}

\maketitle

\begin{abstract}
We prove weak existence and uniqueness in law of $[0,1]$-valued solutions for a class of scalar stochastic heat equations with measurable drifts $b : [0,1] \to \mathbb{R}$, driven by Wright--Fisher noise.
The only assumptions we require on the drift are Borel-measurability, and the one-way inequalities $-Cu\leq b(u)\leq C(1-u)$ for some $C<\infty$, compatible with the necessary condition $b(1)\leq 0\leq b(0)$. These are vastly more general conditions than were previously available. Our proof relies on a stochastic duality between the solution of the stochastic heat equation, and a voting scheme running along the graph of a branching-coalescing particle system, generalising the class of drifts for which a dual process is available.
Our results and their proof put the regularisation-by-noise result of Barnes, Mytnik and Sun, and its explanation, on a much more general footing.
\end{abstract}


\section{Introduction and main result}
We consider the one-dimensional SPDE
\begin{equation}\label{eq:SPDE 1}
\partial_t u = \frac{1}{2} \Delta u + b(u) + \sqrt{\nu u (1 - u)} \dot{W}, 
\end{equation}
defined on a connected spatial domain $E \subseteq \bbR$ with initial condition $u_0 : E \to [0, 1]$ (assumed only to be a measurable function), where $b : [0, 1] \to \Rm$, $\dot{W}$ is space-time white noise, and $\nu > 0$ is a constant. 
We impose one of four different spatial domains and boundary conditions: (i) the whole line, i.e.\ $E=\Rm$, (ii) periodic, i.e.\ $E=\Rm/\Zm$, (iii) $E=[0,1]$ with Dirichlet boundary conditions, $u(0) = u(1) = 0$, or (iv) $E=[0,1]$ with Neumann boundary conditions, $\partial_x u(0) = \partial_x u(1) = 0$.

We are interested in uniqueness in law of solutions to \eqref{eq:SPDE 1} when $b$ is irregular.
Specifically, we will assume that $b:[0,1]\ra \Rm$ is a bounded, measurable function satisfying 
\begin{equation}\label{eq:assumption on tilde b}
    -Cu\leq b(u)\leq C(1-u)\quad\text{for some} \quad C<\infty.
\end{equation}
In order to ensure the well-posedness of $[0,1]$-valued solutions of \eqref{eq:SPDE 1}, it is necessary that $b(0)\geq 0$ and $b(1)\leq 0$, otherwise the solutions could be ``pushed outside of $[0,1]$''. Therefore condition~\eqref{eq:assumption on tilde b} covers almost all drifts for which one could hope to establish the well-posedness of \eqref{eq:SPDE 1}.

We adopt Walsh's theory \cite{MR876085} to define weak solutions of \eqref{eq:SPDE 1} as solutions of the integral equation 
\begin{align}\label{E:MildSol_FKPP}
u_t(x)= \int_{E} p(t,y,x)\,u_0(y)\rmd y &+ \int_0^t\int_{E}p(t-s,y,x)\,b(u_s(y))\,\rmd y\,\rmd s   \notag\\
&+ \int_{E\times [0,t]}p(t-s,y,x)\,
\sqrt{\nu u_s(y) (1-u_s(y))}\,\rmd W(y,s), 
\end{align}
where $p(t,x,y)$  is the transition density of a Brownian motion on $E$ corresponding to our choice of space and boundary condition, i.e.\ (i) Brownian motion on $E=\Rm$, (ii) Brownian motion on $E=\Rm/\Zm$, (iii) Brownian motion on $E=[0,1]$ with instantaneous killing at $\{0, 1\}$, or (iv) Brownian motion on $E=[0,1]$ with reflection at $\{0,1\}$.

Our main result is the following.
\begin{theo}\label{theo:main theorem}
Assume that $b$ is Borel measurable and satisfies \eqref{eq:assumption on tilde b}, $\nu>0$, and any one of the boundary conditions (i)-(iv) hold.
Then for any (possibly random) Borel-measurable initial condition $u_0:E\ra [0,1]$, there exists a weak solution of  \eqref{eq:SPDE 1}, which is unique in law.

\end{theo}
\begin{rmk}\label{rmk:discontinuity difficulties}
Weak existence is classical for continuous drifts \cite{shiga:1988}, but the situation is much more difficult for measurable drifts.
This is because (a) the vanishing noise may preclude a classical argument using Girsanov's theorem, and (b) limiting arguments can break down when discontinuities of $b$ prevent us from passing to a limit.
To our knowledge, the only previous result establishing weak existence with any kind of discontinuity is that of \cite{barnesetal:2025}, which proves weak existence for drifts $b$ given by a power series satisfying certain structural assumptions.
In particular, their power series allow for discontinuities only at one of the endpoints $\{0,1\}$, but nowhere else.
The weak existence part of this theorem should therefore be considered one of the key results of the present article, along with uniqueness in law.
\end{rmk}

We have the following from Theorem \ref{theo:main theorem} and the proof of Lemma \ref{lem:ptwise convergence gives solutions}.
\begin{theo}
Suppose that $b^n$ is Borel measurable and satisfies \eqref{eq:assumption on tilde b} for some fixed $C<\infty$, for all $n<\infty$. Suppose also that $b^n\ra b$ pointwise. Let $u^n$ and $u$ be the weak solutions of \eqref{eq:SPDE 1} with drift $b^n$ and $b$, respectively. Then $u^n\ra u$ locally uniformly in distribution as $n\ra\infty$. 
\end{theo}
This result is difficult for the same reasons as outlined in Remark \ref{rmk:discontinuity difficulties}: the discontinuity of the drifts can interact badly with convergence of solutions. 

\subsection{Background}

The question of the well-posedness of \eqref{eq:SPDE 1} goes back to Shiga \cite{shiga:1988}.
Motivated by population genetics, he proved well-posedness for drifts of the form $b(u)=c_1(1-u)+c_2u+c_3u(1-u)$, $c_1,c_2\geq 0$, $c_3\in \Rm$, using an extremely influential duality argument. The equation \eqref{eq:SPDE 1} has since been studied widely in population genetics for a variety of smooth and irregular drifts, where it arises as the scaling limit of various models \cite{Durrett2016,Fan2021,mueller/tribe:1995,shiga:1988}, and motivated for instance by range expansions, has been widely used to study the effect of noise on front propagation \cite{barnes2024effect,Birzu2018,Etheridge2026,Mueller2011,Mueller:2021}.

The work of \cite{shiga:1988} was generalised in \cite{athreya/tribe:2000} to drifts given by a power series corresponding to the moment generating function of a branching process and, notably, also to more general noise coefficients. Our work complements and extends the more recent works \cite{barnesetal:2025, Mueller:2021}, which established the weak uniqueness of solutions of \eqref{eq:SPDE 1} for various classes of irregular drift functions $b$. More precisely, \cite{Mueller:2021} established uniqueness in law for continuous $b$ satisfying $\lvert b(u)\rvert\leq C\sqrt{u(1-u)}$ all $u\in [0,1]$, for some $C<\infty$, so necessarily $b(0)=b(1)=0$.
The work in \cite{barnesetal:2025}, building on the earlier work of the same authors \cite{Barnes2024}, established uniqueness in law for drifts $b$ having a power series expansion satisfying certain conditions, allowing for drifts which could be non-Lipschitz or even discontinuous at one of the boundary points $\{0, 1\}$, but which are necessarily analytic on $(0,1)$.
They also require structural assumptions on $b$ precluding, for instance, the Allan--Cahn drift $b(u)=u(1-u)(2u-1)$. These are complemented by the very recent work \cite{Mytnik2026} for a similar equation with Feller-type noise, $\sqrt{u}\dot{W}$ (super-Brownian motion), with drift $b$ given by the sum of a step function at $0$ and the Laplace transform of a signed measure. This allows a discontinuity at $0$, but imposes structure and real-analyticity on $(0,1)$. In contrast, our theorem requires only that $b$ is bounded and measurable, along with the one-sided Lipschitz assumption~\eqref{eq:assumption on tilde b}. 

Our technique will be to establish a duality between solutions of \eqref{spde} and a suitable branching-coalescing particle system.
We will use that duality to determine moments $\bbE[u_t(x_1) \ldots u_t(x_N)]$ of the solution, which determine its law \citep[Lemma 1]{athreya/tribe:2000}. This strategy is well-established in SPDE models \citep{athreya/tribe:2000, barnesetal:2025, shiga:1988}.
The duality considered in the cited works can be viewed as an SPDE extension of the classical duality between branching Brownian motion and the FKPP equation \citep{McKean1975,Skorokhod1964}, and consequently they require the drift to be a moment generating function of a branching process. 

The work in \cite{etheridgeetal:2017} instigated growth in the study of moment duals for PDEs and SDEs whose drift term (corresponding to the zeroth order growth term in a PDE context) cannot be viewed as the moment--generating function of a branching process.
They considered the deterministic Allen--Cahn equation and established a moment dual in terms of a ternary branching Brownian motion with the key addition of a voter model along its branches.
This has been extended to moment duality for a general family of analytic drifts \cite{An2023}.
However, the branching property is fundamental to the construction of this duality.
This represents a fundamental obstacle to constructing the same duality representation for Wright--Fisher-type SPDEs using branching-coalescing Brownian motions, as coalescence destroys the branching property.

A similar approach has been employed for non-spatial Wright--Fisher SDEs using Bernstein polynomials \citep{corderoetal:2022, Hernandez2026, koskelaetal:2025}, and for interacting particle systems \citep{Huang2021}.
We extend these to an SPDE setting, with our approach being an SPDE extension of \cite{koskelaetal:2025} in particular.
We use a voting scheme based on a Bernoulli factory as the branching mechanism for particles, which makes it possible to handle non-differentiable drift functions.
This has at its core the so-called Bernstein duality, which was introduced for non-spatial models by \cite{corderoetal:2022}.
A Bernoulli factory for a function $f: [0,1] \to [0,1]$ amounts to a series expansion for $f$ with probabilistically interpretable, non-negative coefficients; see \eqref{bf1} below.
Hence, they are natural objects with which to replace drift terms given by branching process moment-generating functions when the branching property is not available.
We will further extend this to discontinuous drifts by showing that the value of our duality representations converge under pointwise limits of the drifts, notwithstanding that the dual branching process process may not itself converge as there is no Bernoulli factory for discontinuous drifts (see Theorem \ref{thm:bernoulli general drift}).

Duality is fundamental to the stochastic heat equation with Wright--Fisher noise.
Applications go well beyond uniqueness: for instance, duality is used to understand associated genealogies, establish the convergence of discrete models to it, to understanding its long-time behaviour, and even to understand coalescing Brownian motion itself \cite{Barnes2024,Blath2023,Durrett2016,Fan2023,Mueller2011}.
Our proof puts this duality on a much more general footing, which could have a variety of potential applications.

If we take a drift $b$ satisfying \eqref{eq:assumption on tilde b}, then the corresponding Wright--Fisher SDE,
\[
\rmd X_t=b(X_t)\rmd t+\sqrt{X_t(1-X_t)}\rmd W_t,
\]
may not be unique in law.
For instance, this is the case for $b(u)=\mathbbm{1}(u>0)-u$, as observed in \cite{barnesetal:2025}.
This can be seen by noting that Feller's criterion implies the existence of a solution to the above SDE which starts from $0$ but then immediately goes above $0$, in addition to the trivial solution starting from $0$.

In contrast, Theorem \ref{theo:main theorem} ensures that the SPDE \eqref{eq:SPDE 1} has unique-in-law weak solutions.
This is the so-called \textit{regularisation by noise} phenomenon studied in \cite{barnesetal:2025}.
Their explanation for it is that, in the SPDE setting, one can construct moment duals based on branching processes with non-integrable reproduction law (or even branching into infinitely many particles), where the corresponding duals would be ill-defined in the SDE or PDE setting.
Such non-integrable reproduction laws permit moment duals for more irregular drifts.

At the heart of the matter is the fact that non-integrable reproduction laws allow for the number of particles to explode in finite time.
The moment duals in the PDE setting involve branching processes without any coalescence, so once the number of particles explodes we are stuck with infinitely many particles.
In the SDE setting, the particles coalesce with each other at Poisson rate, i.e.\ like Kingman's coalescent, so the coalescence will cause the number of particles to come down from infinity like $1/t$.
However, if we take for simplicity the case of each particle branching into infinitely many particles at a constant Poisson rate $k_{\infty}>0$, then \cite{kyprianouetal:2017} found a counterintuitive phase transition at $k_{\infty}=\nu/2$.
For $k_{\infty}<\nu/2$, the dual explodes but comes down from infinity, so that the number of particles is finite at almost all times.
For $k_{\infty}\geq\nu/2$, the dual is absorbed at infinity.
A similar and complementary logarithmic moment condition on family sizes of the branching process was established by \cite{lambert:2005}.
Therefore, a well-defined dual can only be obtained for a restricted family of drifts for which the branching-coalescing particle system never explodes, or at least comes down from infinity.
In contrast, \cite{Barnes2024, barnesetal:2025} established that in the SPDE setting, infinitely many particles at a single point will come down from infinity like $1/\sqrt{t}$.
Since 
\begin{equation*}
\int_0^{\epsilon}\frac{1}{\sqrt{t}}\rmd t<\infty
\end{equation*}
for all $\epsilon>0$, the system always comes down from infinity and explosion times are isolated points.

The heuristic reason for this different rate of coming down from infinity is that in the SPDE setting, particles are Brownian and coalesce according to their intersection local time.
For any pair of paricles, their intersection local time is equal to the local time process $L_t$ of a single rate $2$ Brownian motion at $0$ (the Brownian motion given by their difference).
This local time process is equal in law to the running maximum of a single Brownian motion started at $0$, hence $L_t\approx \sqrt{t}$ for small $t$.
Since this is much larger than $t$ for small $t$, particles coalesce much faster than is the case for Kingman's coalescent, which describes the dual in the SDE setting.

Despite regularisation by noise, the duality employed by \cite{barnesetal:2025} only permits one to consider $b$ which are the moment generating functions of branching processes. As we previously stated, they may be irregular (or even discontinuous) at one of the endpoints $\{0,1\}$, but must be analytic on $(0,1)$.
To see that moment duality for the regularisation by noise phenomenon applies more generally, one needs to construct a dual for a more general class of irregular $b$.
We do just that, with the duals we construct for $b$ satisfying \eqref{eq:assumption on tilde b} relying on the ``fast coming down from infinity'' established in \cite{barnesetal:2025} to ensure their well-posedness.
This puts their explanation for regularisation by noise on a much more general footing.

Finally, we note that pathwise uniqueness remains a very longstanding open problem for the stochastic heat equation with Wright-Fisher noise, even in the absence of drift.
It is unclear, at least to us, what answer one should even expect here.

\subsection*{Proof structure}
The rest of this paper is devoted to the proof of Theorem \ref{theo:main theorem}. We will impose \eqref{eq:assumption on tilde b} on the drift throughout, but with successively weaker assumptions on the regularity of $b$. 

We firstly construct a duality representation for solutions of \eqref{eq:SPDE 1} with continuous drifts in Section \ref{sec:cts}. We then extend this duality to Baire class-1 drifts, which we use to establish well-posedness in Section \ref{sec:Baire 1}. The key point will be to show convergence of the value of the duality expressions, but not the dual itself or SPDE as processes, under pointwise limits of the drift. Finally, we extend this to Borel-measurable drifts in Section \ref{sec:measurable}.

\section{Duality for continuous drifts}\label{sec:cts}

Throughout this section, we will assume that $b$ is continuous and satisfies \eqref{eq:assumption on tilde b}, and that $u$ is an arbitrary weak solution of \eqref{eq:SPDE 1} for one of the chosen boundary conditions (i)-(iv), with initial condition $u_0$: a (possibly random) Borel-measurable function $E\ra [0,1]$.

The first order of business is to rewrite the form of the drift we consider.
\begin{lem}\label{lem:equivalence of different formulations of b}
    The following are equivalent:
    \begin{enumerate}
        \item $b:[0,1]\rightarrow\Rm$ is continuous and satisfies $-Cu\leq b(u)\leq C (1-u)$ for all $u\in [0,1]$, for some $C<\infty$;
        \item there exists $\mu> 0$ and $\tilde{b}:[0,1]\ra [0,1]$ such that $b(u)=\mu(\tilde{b}(u)-u)$, where $\tilde{b}$ satisfies the conditions of \cite{keane/obrien:1994} for existence of a Bernoulli factory, meaning that it is continuous, and either constant or \emph{polynomially bounded}:
\begin{equation}\label{eq:polynomial boundedness}
\min\{ \tilde{b}(u), 1 - \tilde{b}(u) \} \geq \min\{ u, 1 - u \}^n, \\
\end{equation}
for all $u \in [0, 1]$ and some $n > 0$.
    \end{enumerate}
\end{lem}
\begin{proof}[Proof of Lemma \ref{lem:equivalence of different formulations of b}]
    We firstly prove that $1$ implies $2$. We define $\tilde{b}(u)=b(u)/\mu+u$ (so that $b(u)=\mu(\tilde{b}(u)-u)$) for some $\mu>0$ to be determined, so that $\tilde{b}$ is necessarily continuous for any $\mu>0$. Then we observe that
    \[
    \tilde{b}(u)\geq -\frac{C}{\mu}u+u\geq \frac{u}{2}\geq \min\{u,1-u\}^2,
    \]
    for all $\mu\geq 2C$. 
    On the other hand,
    \[
    1-\tilde{b}(u)\geq 1-\frac{C}{\mu}(1-u)-u\geq \frac{1}{2}(1-u)\geq \min\{u,1-u\}^2,
    \]
    for all $\mu\geq 2C$. Therefore $1$ implies $2$.
    
    We now prove that $2$ implies $1$. 
    Since $\tilde{b}(u) \in [0, 1]$, we have that $b(u) = \mu (\tilde{b}(u) - u)$ satisfies
    \[
    -\mu u=\mu(0-u)\leq b(u)\leq \mu (1-u).
    \]
    Therefore $1$ is satisfied with $C=\mu$.
\end{proof}

In light of Lemma \ref{lem:equivalence of different formulations of b}, we henceforth consider the one-dimensional SPDE
\begin{equation}
\partial_t u = \frac{1}{2} \Delta u + \mu (\tilde{b}(u) - u), + \sqrt{\nu u (1 - u)} \dot{W}, \label{spde}
\end{equation}
where $\tilde{b} : [0, 1] \to [0, 1]$ is assumed to satisfy \eqref{eq:polynomial boundedness}, $\dot{W}$ is space-time white noise, $\mu, \nu > 0$ are constants, and $u$ satisfies one of the boundary conditions (i)-(iv) described in the introduction. 

Since $\tilde{b}$ satisfies the conditions of \cite{keane/obrien:1994} for the existence of a Bernoulli factory, there exists a sequence of continuous and polynomially bounded functions $\{ \tilde{b}_k \}_{k \geq 1}$, as well as positive integers $\{ \eta_k \}_{k \geq 1}$ such that
\begin{equation}
\tilde{b}(u) = \sum_{k = 1}^{\infty} \Big( \frac{3}{4} \Big)^{k - 1} \frac{1}{4} \bbP( \tilde{b}_k( Y( \eta_k, u) / \eta_k ) \geq 1/2 ), \label{bf1} \\
\end{equation}
where $Y(n, u) \sim \text{Bin}(n, u)$.

For $N \in \bbN$ let $[N] := \{1, \ldots, N\}$.
We will construct a dual interacting particle system as follows.
Let $I_t$ be the number of particles alive at time $t \geq 0$.
The position of the particle with label $i \in [I_t]$ is $X_t^i$.
At time $t = 0$ there are $N$ particles at positions $x^{1:N} := (x^1, \ldots, x^N)$.
Each particle undergoes an independent Brownian motion on $E$, and is either periodic, killed, or reflecting the boundaries of $E$ (if any), matching the boundary condition of \eqref{spde}.
During its lifetime, each particle undergoes two types of events:
\begin{enumerate}
\item Branching: each particle dies at rate $\mu \rmd t$ and leaves $Z_1$ offspring at its location, with $\bbP( Z_1 = \eta_k ) = ( 3 / 4 )^{ k - 1 } (1 / 4)$.
\item Coalescence: particles $X_t^i$ and $X_t^j$ merge into one particle at rate $(\nu / 2) \rmd L_t^{i, j}$, where $L_t^{i, j}$ is the local time of $X_t^i - X_t^j$ at zero.
\end{enumerate}
The particle system is initialised at locations $x^{1:N}$ and evolves for $t$ units of time.
At that point the remaining particles $\{ X_t^i : i \in [I_t] \}$ are assigned binary types, 0 or 1, independently and with respective probabilities $1 - u_0( X_t^i )$ and $ u_0(X_t^i)$.
Types are propagated along the edges of the random network traced by the particle system.
At a coalescence event, the two coalescing particles inherit the type of the particle to which they merge.
At a branching event with $\eta_k$ children, the parent particle is assigned type 1 if $\tilde{b}_k(\#\{\text{type } 1 \text{ children}\} / \eta_k) \geq 1/2$.
With probability 1, this particle system contains only finitely many branching events and particles in any finite time-window whenever $N < \infty$, since $\eta_k < \infty$ for each $k \in \bbN$ \cite[Theorem 1.4]{barnesetal:2025}.

\begin{rmk}\label{rmk:applying Barnes to other bcs}
The four choices of boundary condition (i)-(iv) corresponds, respectively, to Brownian motion in $E=\Rm$, Brownian motion on $E=\Rm/\Zm$, Brownian motion on $E=[0,1]$ with instantaneous killing at at $\{0,1\}$, and Brownian motion on $E=[0,1]$ with reflection at $0,1$. On the other hand, \cite{barnesetal:2025} only deals with the first of these. However, we can (ii) map Brownian motion on $\Rm$ to Brownian motion on $\Rm/\Zm$ by mapping $x\mapsto x$ (mod $1$), (iii) kill Brownian motion immediately upon hitting $(0,1)^c$, or (iv) map Brownan motion on $\Rm$ to reflected Brownian motion on $[0,1]$ using 
\[
x\mapsto \begin{cases}
    h,\quad x=2N+h, \; h\in [0,1)\\
    1-h,\quad x=2N+1+h,\; h\in [0,1)
\end{cases}
\]
These allow us to map branching-coalescing Brownian motion (BCBM) on $\Rm$ to BCBM with (ii)-(iv) in a manner that only reduces the number of particles, so \cite{barnesetal:2025} transfers to (ii)-(iv).
\end{rmk}

For a set $A$, let $\Pi_j(A)$ be the set of subsets of size $j$.
Let $\tau^N$ denote the set of binary strings of length $N$, and let $\tau_{\ell}^N$ be the subset of strings which sum to $\ell$.
To adapt the Bernstein duality of \cite{corderoetal:2022} to our spatial model, let $P_t^{u_0}(x^{1:N} | H)$ be the conditional probability that all $N$ leaves (time 0 particles) of the system are of type 1, given the realisation of the branching-coalescing particle system $H = ( H_s )_{s \in [0, t]}$.
For concreteness we order the particles by time of birth, with the ordering of siblings chosen arbitrarily.
In the absence of branching events $P_t^{u_0}(x^{1:N} | H) = \prod_{i \in I_t} u_0(X_t^i)$, while when $N = 1$, a single branching event into $\eta$ ancestors with no further events yields
\begin{equation*}
P_t^{u_0}(x | H) = \sum_{j = 0}^{\eta} \mathds{1}\{\tilde{b}_k(j/\eta) \geq 1/2 \} \sum_{(i_1, \ldots, i_j) \in \Pi_j([\eta])} \Bigg( \prod_{k = 1}^j u_0(X_t^{i_k}) \Bigg) \prod_{i \neq i_1, \ldots, i_j} ( 1 - u_0(X_t^i)).
\end{equation*}
Other patterns of branching and coalescence events result in a more complicated expression of the form
\begin{equation*}
P_t^{u_0}(x^{1:N} | H) = \sum_{q \in \tau^{I_t}} V_t^{I_t}(q) \Bigg( \prod_{k : q_k = 1} u_0(X_t^k) \Bigg) \prod_{k : q_k = 0} (1 - u_0(X_t^k)),
\end{equation*}
where $V_t^{I_t}(q)$ is the binary indicator that all $N$ leaves are of type 1 given that the roots with indices with $q_k = 1$ are type 1, while all others are of type 0.
Implicitly it depends on the whole history of the particle system, but that dependence has been suppressed for readability.

Let $x^{1:N \setminus i} = (x^1, \ldots, x^{i - 1}, x^{i + 1}, \ldots, x^N)$ and $x^{1:N} \oplus_j x = (x^1, \ldots, x^N, x, \ldots, x)$, where $j$ is the number of times the scalar $x$ is concatenated onto $x^{1:N}$.
Similarly, for two vectors $x^{1:N}, y^{1:M}$ let $x^{1:N} \oplus y^{1:M} = (x^1, \ldots, x^N, y^1, \ldots, y^M)$.
For $q \in \tau^N$ and $i \neq j \in [N]$ we define the $(i, j)$-split operator $S_{i, j}$ as
\begin{equation}\label{split_operator}
    S_{i, j} q := q^{1:((i \wedge j) - 1)} \oplus_1 q^N \oplus q^{(i \wedge j):((i \vee j) - 1)} \oplus_1 q^N \oplus q^{(i \vee j) : (N - 1)},
\end{equation}
while for $i \in [N]$ and $k \geq 1$ such that $i + \eta_k - 1 \leq N$, we define the $(i, k)$-merger operator $M_{i, k}$ as
\begin{equation}\label{merger_operator}
    M_{i, k} q = q^{1:(i - 1)} \oplus_1 \mathds{1}\{\tilde{b}_k(\|q^{(N - \eta_k + 1):N}\|_1 / \eta_k) \geq 1/2\} \oplus q^{(i + 1):(N - \eta_k)}.
\end{equation}
The collection $(X_t^{1:I_t}, V_t^{I_t}) \equiv (X_t^{1:I_t}, (V_t^{I_t}(q))_{q \in \{0, 1\}^{|I_t|}})$ is a Markov process in which each $X_t^i$ undergoes an independent Brownian motion, and the process jumps from state $(x^{1:N}, v^N)$ to
\begin{enumerate}
    \item $(x^{1:N\setminus i} \otimes_{\eta_k} x^i, (v^N(M_{i, k} q))_{q \in \tau^{N + \eta_k - 1}})$ with rate $\mu (3 / 4)^{k - 1} (1 / 4) \rmd t$,
    \item $(x^{1:N\setminus \{i, j\}} \oplus_1 x^i, (v^N(S_{i, j} q))_{q \in \tau^{N - 1}})$ with rate $(\nu / 2) \rmd L_t^{i, j}$.
\end{enumerate}

For a function $u : [0, 1] \to [0, 1]$ let
\begin{equation*}
H(u; x^{1:N}) := \Bigg( \Big(\prod_{k : q_k = 1} u(x^k) \Big) \prod_{k : q_k = 0} (1 - u(x^k)) \Bigg)_{q \in \tau^N},
\end{equation*}
and let
\begin{equation*}
\langle u, v \rangle := \sum_{i = 1}^N u_i v_i
\end{equation*}
be the usual inner product for two vectors of equal length $N$.
Subscripts on expected values denote initial conditions, and also implicitly define the randomness with respect to which an expectation is taken: $\bbE_{u_0}[\cdot]$ denotes an expectation with respect to the law of the solution of \eqref{eq:SPDE 1} with initial condition $u_0(\cdot)$, while $\bbE_{(x^{1:N}, v^N)}[\cdot]$ is an expectation with respect to the law of the branching-coalescing voting model described above with $N$ initial particles at locations $x^{1:N}$.

\begin{theo}\label{thm:bernoulli}
Let $\tilde{b} : [0, 1] \to [0, 1]$ in \eqref{spde} be continuous and polynomially bounded.
For $N \in \mathbb{N}$, $x^{1:N} \in [0, 1]^N, v^N \in \{0, 1\}^{2^N}$, and any measurable $u_0 : [0, 1] \to [0, 1]$, 
\begin{equation}\label{main_eq}
\bbE_{u_0}[ \langle H(u_t; x^{1:N}), v^N \rangle ] = \bbE_{(x^{1:N}, v^N)}[\langle H(u_0; X_t^{1:I_t}), V_t^{I_t} \rangle].
\end{equation}
\end{theo}
\begin{rmk}
    Uniqueness in law follows by choosing $v^N(q) = 1$ when all entries of $q$ are 1, and $v^N(q) = 0$ otherwise.
    This yields the moment $\bbE[u_t(x^1) \ldots u_t(x^N)]$ on the left-hand side of \eqref{main_eq}, so that weak uniqueness holds by \cite[Lemma 1]{athreya/tribe:2000}.
    The same also applies to Theorem \ref{thm:bernoulli general drift} for Baire class-1 drifts, and to the result for measurable drifts in Section \ref{sec:measurable}.
\end{rmk}

\begin{proof}[Proof of Theorem \ref{thm:bernoulli}]
For $\varepsilon > 0$, let $p_{\varepsilon}$ be the Gaussian density with mean zero and variance $\varepsilon$, and for a function $f : E \to \bbR$ let $f^{\varepsilon}$ be the convolution of $f$ with the Gaussian density with mean 0 and variance $\varepsilon$:
\begin{equation*}
    f^{\varepsilon}(x) := (f * p_{\varepsilon})(x) = \frac{1}{\sqrt{2 \pi \varepsilon}} \int_E f(y) e^{-(y - x)^2 / (2 \varepsilon)} \rmd y.
\end{equation*}
By It\=o's formula,
\begin{align}
&\bbE_{u_0}[\langle H(u_t^{\varepsilon}; x^{1:N}), v^N \rangle] - \langle H(u_0^{\varepsilon}; x^{1:N}), v^N \rangle \notag \\
&= \int_0^t \Bigg( \bbE_{u_0}\Bigg[\sum_{q \in \tau^N} v^N(q) \sum_{i = 1}^N (-1)^{\mathds{1}\{q_i = 0\}} \Big( \prod_{j \neq i : q_j = 0} (1 - u_s^{\varepsilon}(x^j)) \Big) \Big( \prod_{j \neq i : q_j = 1} u_s^{\varepsilon}(x^j) \Big) \frac{1}{2} \Delta u_s^{\varepsilon}(x^i) \Bigg] \notag \\
&\phantom{=} + \mu \bbE_{u_0}\Bigg[\sum_{q \in \tau^N} v^N(q) \sum_{i = 1}^N (-1)^{\mathds{1}\{q_i = 0\}} \Big( \prod_{j \neq i : q_j = 0} (1 - u_s^{\varepsilon}(x^j)) \Big) \Big( \prod_{j \neq i : q_j = 1} u_s^{\varepsilon}(x^j) \Big) \notag \\
&\phantom{= + \mu \bbE_{u_0}\Bigg[\sum_{q \in \tau^N} v^N(q) \sum_{i = 1}^N (-1)^{\mathds{1}\{q_i = 0\}} \Big( \prod_{j \neq i : q_j = 1} u_s^{\varepsilon}(x^j)} \times ( (\tilde{b} \circ u_s)^{\varepsilon}(x^i) - u_s^{\varepsilon}(x^i) ) \Bigg] \notag \\
&\phantom{=} + \frac{\nu}{2} \bbE_{u_0}\Bigg[\sum_{q \in \tau^N} v^N(q) \sum_{i \neq j}  (-1)^{\mathds{1}\{q_i = 0\}} (-1)^{\mathds{1}\{q_j = 0\}} \Big( \prod_{k \neq i, j : q_k = 0} (1 - u_s(x^k) ) \Big) \notag \\
&\phantom{= + \frac{\nu}{2} \bbE_{u_0}} \times \Big( \prod_{k \neq i, j : q_k = 1} u_s(x^k) \Big) \int_E p_{\varepsilon}(y - x^i) p_{\varepsilon}(y - x^j) u_s(y) (1 - u_s(y)) \rmd y \Bigg] \Bigg) \rmd s \notag \\
&=: \int_0^t \Big( \bbE_{u_0}[A_s^{\varepsilon}(x^{1:N}, v^N)] + \mu \bbE_{u_0}[B_s^{\varepsilon}(x^{1:N}, v^N)] + \frac{\nu}{2}\bbE_{u_0}[C_s^{\varepsilon}(x^{1:N}, v^N)] \Big) \rmd s. \label{main_lhs}
\end{align}

\begin{rmk}\label{rmk:where b enters proof from SPDE}
    For future convenience in the proof of Theorem \ref{thm:bernoulli general drift}, we highlight that \eqref{main_lhs} is the only place where $\tilde b$ enters our proof as the drift term in \eqref{spde}.
\end{rmk}

For an arbitrary twice-differentiable function $h : [0,1] \to [0,1]$, the right-hand side of \eqref{main_eq} can be written as
\begin{align}
&\bbE_{(x^{1:N}, v^N)}[\langle H(h; X_t^{1:I_t}), V_t^{I_t} \rangle] - \langle H(h; x^{1:N}), v^N \rangle \notag \\
&= \int_0^t \bbE_{(x^{1:N}, v^N)}\Bigg[ \sum_{i = 1}^{I_s} \sum_{q \in \tau^{I_s}} V_s^{I_s}(q) (-1)^{\mathds{1}\{q_i = 0\}} \Big( \prod_{j \neq i : q_j = 1} h(X_s^j) \Big) \notag \\
&\phantom{= \int_0^t \bbE_{(x^{1:N}, v^N)}\Bigg[ \sum_{i = 1}^{I_s} \sum_{q \in \tau^{I_s}} V_s^{I_s}(q)} \times \Big( \prod_{j \neq i : q_j = 0} (1 - h(X_s^j)) \Big) \frac{1}{2} \Delta h(X_s^i) \Bigg] \rmd s \notag \\
&\phantom{=} + \mu \int_0^t \Bigg( \sum_{k = 1}^{ \infty} \Big( \frac{3}{4} \Big)^{k - 1} \frac{1}{4} \bbE_{(x^{1:N}, v^N)}\Bigg[\sum_{i = 1}^{I_{s-}} \sum_{q \in \tau^{I_{s-} - 1}} \sum_{\ell = 0}^{\eta_k} \sum_{\tilde{q} \in \tau_{\ell}^{\eta_k}} V_{s-}^{I_{s-}}(M_{i, k} (q \oplus \tilde{q})) \notag \\
&\phantom{= + \mu \int_0^t \Bigg( \sum_{k = 1}^{ \infty}} \times \Big( \prod_{j \neq I_{s-} : q_j = 1} h(X_{s-}^{j + \mathds{1}\{j \geq i\}}) \Big) \Big( \prod_{j \neq I_{s-} : q_j = 0} (1 - h(X_{s-}^{j + \mathds{1}\{j \geq i\}})) \Big) \notag \\
&\phantom{= + \mu \int_0^t \Bigg( \sum_{k = 1}^{ \infty} \frac{1}{4} \Big( \frac{3}{4} \Big)^{k - 1} \bbE_{(x^{1:N}, v^N)}\Bigg[} \times h(X_{s-}^i)^{\ell} (1 - h(X_{s-}^i))^{\eta_k - \ell} \Bigg] \notag \\
&\phantom{= +} - \bbE_{(x^{1:N}, v^N)}\Bigg[ I_{s-} \sum_{q \in \tau^{I_{s-}}} V_{s-}^{I_{s-}}(q) \Big( \prod_{j : q_j = 1} h(X_{s-}^j) \Big) \Big( \prod_{j : q_j = 0} (1 - h(X_{s-}^j)) \Big) \Bigg] \Bigg) \rmd s \notag \\
&\phantom{=} + \frac{\nu}{2} \bbE_{(x^{1:N}, v^N)}\Bigg[\int_0^t \Bigg( \sum_{i = 1}^{I_{s-}} \sum_{j \neq i} \sum_{q \in \tau^{I_{s-} - 1}} V_{s-}^{I_{s-}}(S_{i,j} q) \notag \\
&\phantom{= + \frac{\nu}{2} \bbE_{(x^{1:N}, v^N)}\Bigg[ \int_0^t \Bigg( \sum_{i = 1}^{I_{s-}} \sum_{j \neq i}} \times \Big( \prod_{k < I_{s-} - 1 : q_k = 1} h(X_{s-}^{k + \mathds{1}\{k \geq i \wedge j\} + \mathds{1}\{k \geq i \vee j\}}) \Big) \notag \\
&\phantom{= + \frac{\nu}{2} \bbE_{(x^{1:N}, v^N)}\Bigg[ \int_0^t \Bigg( \sum_{i = 1}^{I_{s-}} \sum_{j \neq i}} \times \Big( \prod_{k < I_{s-} - 1 : q_k = 0} (1 - h(X_{s-}^{k + \mathds{1}\{k \geq i \wedge j\} + \mathds{1}\{k \geq i \vee j\}})) \Big) \notag \\
&\phantom{= + \frac{\nu}{2} \bbE_{(x^{1:N}, v^N)}\Bigg[ \int_0^t \Bigg( \sum_{i = 1}^{I_{s-}} \sum_{j \neq i}} \times h(X_{s-}^i)^{\mathds{1}\{q_{I_{s-} - 1} = 1\}} (1 - h(X_{s-}^i))^{\mathds{1}\{q_{I_{s-} - 1} = 0\}} \notag \\
&\phantom{= +} - I_{s-} (I_{s-} - 1) \sum_{q \in \tau^{I_{s-}}} V_{s-}^{I_{s-}}(q) \Big( \prod_{j : q_j = 1} h(X_{s-}^j) \Big) \Big( \prod_{j : q_j = 0} (1 - h(X_{s-}^j)) \Big) \Bigg) \rmd L_s^{i, j} \Bigg] \notag \\
&=: \int_0^t \Big( \bbE_{(x^{1:N}, v^N)}[\bar{A}_s(h)] + \mu \bbE_{(x^{1:N}, v^N)}[\bar{B}_s(h)] \Big) \rmd s + \frac{\nu}{2} \bbE_{(x^{1:N}, v^N)}\Big[\int_0^t \rmd \bar{C}_s(h)\Big]. \label{main_rhs}
\end{align}

Let $u_t^{\varepsilon}$ and $(X_t^{1:I_t}, V_t^{I_t})$ be independent copies of a solution of \eqref{spde} and the labelled branching-coalescing particle system, respectively, and let
\begin{equation*}
    g(t, s, \varepsilon) := \bbE_{(u_0; x^{1:N}, v^N)}\Bigg[\sum_{q \in \tau^{I_s}} V_s^{I_s}(q) \Big( \prod_{i : q_i = 1} u_t^{\varepsilon}(X_s^i) \Big) \Big( \prod_{i : q_i = 0} (1 - u_t^{\varepsilon}(X_s^i)) \Big) \Bigg],
\end{equation*}
where the expectation is with respect to the law of both processes and dependence on the initial conditions $(u_0; x^{1:N}, v^N)$ is suppressed on the left-hand side.
Then for $T<\infty$,
\begin{align}
    &\int_0^T [g(r, 0, \varepsilon) - g(0, r, \varepsilon)] \rmd r \notag\\
    &= \int_0^T [g(T - r, r, \varepsilon) - g(0, r, \varepsilon)] \rmd r - \int_0^T [g(r, T - r, \varepsilon) - g(r, 0, \varepsilon)] \rmd r.\label{eq:g decomposition}
\end{align}
Substituting in \eqref{main_lhs} and \eqref{main_rhs} yields
\begin{align}
    &\int_0^T [g(r, 0, \varepsilon) - g(0, r, \varepsilon)] \rmd r \notag \\
    &= \int_0^T \int_0^{T - r} \bbE_{(u_0; x^{1:N}, v^N)}[A_s^{\varepsilon}(X_r^{1:I_r}, V_r^{I_r}) - \bar{A}_s(u_r^{\varepsilon})] \rmd s \rmd r \notag \\
    &\phantom{=} + \mu \int_0^T \int_0^{T - r} \bbE_{(u_0; x^{1:N}, v^N)}[ B_s^{\varepsilon}(X_r^{1:I_r}, V_r^{I_r}) - \bar{B}_s(u_r^{\varepsilon})] \rmd s \rmd r \notag \\
    &\phantom{=} + \frac{\nu}{2} \int_0^T \bbE_{(u_0; x^{1:N}, v^N)}\Bigg[ \int_0^{T - r} [ C_s^{\varepsilon}(X_r^{1:I_r}, V_r^{I_r}) \rmd s - \rmd \bar{C}_s(u_r^{\varepsilon})] \Bigg] \rmd r. \label{process_difference}
\end{align}
We will consider each of the three summands on the right-hand side in turn.
Firstly,
\begin{align}
     &A_s^{\varepsilon}(X_r^{1:I_r}, V_r^{I_r}) - \bar{A}_s(u_r^{\varepsilon}) \notag\\
    &=\sum_{q \in \tau^{I_r}} \sum_{i = 1}^{I_r} V_r^{I_r}(q) (-1)^{\mathds{1}\{q_i = 0\}} \Big( \prod_{j \neq i : q_j = 1} u_s^{\varepsilon}(X_r^j) \Big) \Big( \prod_{j \neq i : q_j = 0} (1 - u_s^{\varepsilon}(X_r^j)) \Big) \frac{1}{2} \Delta u_s^{\varepsilon}(X_r^i) \notag\\
    &\phantom{=} - \sum_{q \in \tau^{I_s}} \sum_{i = 1}^{I_s} V_s^{I_s}(q) (-1)^{\mathds{1}\{q_i = 0\}} \Big( \prod_{j \neq i : q_j = 1} u_r^{\varepsilon}(X_s^j) \Big) \Big( \prod_{j \neq i : q_j = 0} (1 - u_r^{\varepsilon}(X_s^j)) \Big) \frac{1}{2} \Delta u_r^{\varepsilon}(X_s^i).\notag
\end{align}
The population size $I_t$ has an exponential moment by \cite[Proposition 2.1]{barnesetal:2025}, and hence the random numbers of terms in the summations over $q \in \tau^{I_t}$ and $i \in [I_t]$, for $t \in \{r, s\}$ have finite mean.
All other terms in the integrands are bounded, so antisymmetry of $\{(s,r)\in \Rm_{>0}:s+r\leq T\}$ yields that
\begin{equation*}
    \int_0^T \int_0^{T - r} \bbE_{(u_0; x^{1:N}, v^N)}[ A_s^{\varepsilon}(X_r^{1:I_r}, V_r^{I_r}) - \bar{A}_s(u_r^{\varepsilon})] \rmd s \rmd r = 0
\end{equation*}
for any $\varepsilon > 0$.
We now take $\varepsilon \to 0$ in \eqref{process_difference}, leaving
\begin{align}
    &\int_0^T [g(r, 0, 0) - g(0, r, 0)] \rmd r \notag \\
    &= \mu \int_0^T \int_0^{T - r} \bbE_{(u_0; x^{1:N}, v^N)}[ B_s^0(X_r^{1:I_r}, V_r^{I_r}) - \bar{B}_s(u_r)] \rmd s \rmd r \notag \\
    &\phantom{=} + \frac{\nu}{2} \int_0^T \bbE_{(u_0; x^{1:N}, v^N)}\Bigg[ \int_0^{T - r} [ C_s^{\varepsilon}(X_r^{1:I_r}, V_r^{I_r}) \rmd s - \rmd \bar{C}_s(u_r^{\varepsilon})] \Bigg] \rmd r. \label{eps_to_zero}
\end{align}
The integrand $B_s^{0}(X_r^{1:I_r}, V_r^{I_r})$ on the right-hand side expands into
\begin{align*}
&B_s^{0}(X_r^{1:I_r}, V_r^{I_r}) \\
&= \sum_{q \in \tau^{I_r}} V_r^{I_r}(q) \Bigg(\sum_{i : q_i = 1 } \Big( \prod_{j \neq i : q_j = 1} u_s(X_r^j) \Big) \Big( \prod_{j : q_j = 0} (1 - u_s(X_r^j)) \Big) [ \tilde{b}(u_s(X_r^i)) - u_s(X_r^i) ] \\
&\phantom{= } - \sum_{i : q_i = 0} \Big( \prod_{j : q_j = 1} u_s(X_r^j) \Big) \Big( \prod_{j \neq i : q_j = 0} (1 - u_s(X_r^j)) \Big) [ \tilde{b}(u_s(X_r^i)) - 1 + 1 - u_s(X_r^i) ] \Bigg) \\
&= \sum_{q \in \tau^{I_r}} V_r^{I_r}(q) \sum_{i : q_i = 1 } \Big( \prod_{j \neq i : q_j = 1} u_s(X_r^j) \Big) \Big( \prod_{j : q_j = 0} (1 - u_s(X_r^j)) \Big) \tilde{b}(u_s(X_r^i)) \\
&\phantom{=} + \sum_{q \in \tau^{I_r}} V_r^{I_r}(q) \sum_{i : q_i = 0} \Big( \prod_{j : q_j = 1} u_s(X_r^j) \Big) \Big( \prod_{j \neq i : q_j = 0} (1 - u_s(X_r^j)) \Big) [ 1 - \tilde{b}(u_s(X_r^i)) ] \rmd s \\
&\phantom{=} - I_r \sum_{q \in \tau^{I_r}} V_r^{I_r}(q) \Big( \prod_{j : q_j = 1} u_s(X_r^j) \Big) \Big( \prod_{j : q_j = 0} (1 - u_s(X_r^j)) \Big),
\end{align*}
and using \eqref{bf1},
\begin{align}
&B_s^{0}(X_r^{1:I_r}, V_r^{I_r}) \notag \\
&= \sum_{k = 1}^{\infty} \frac{3^{k - 1}}{4^k} \sum_{\ell = 0}^{\eta_k} \mathds{1}\Big\{ \tilde{b}_k\Big( \frac{\ell}{\eta_k} \Big) \geq \frac{1}{2} \Big\} \binom{\eta_k}{\ell} \sum_{q \in \tau^{I_r}} V_r^{I_r}(q) \sum_{i : q_i = 1} \Big( \prod_{j \neq i : q_j = 1} u_s(X_r^j) \Big) \notag \\
&\phantom{= \sum_{k = 1}^{\infty} \Big(\frac{3}{4}\Big)^{k - 1} \frac{1}{4} \sum_{\ell = 0}^{\eta_k} \mathds{1}\{ \tilde{b}_k\}} \times  \Big( \prod_{j : q_j = 0} (1 - u_s(X_r^j)) \Big) u_s(X_r^i)^{\ell} (1 - u_s(X_r^i))^{\eta_k - \ell} \notag \\
&\phantom{=} + \sum_{k = 1}^{\infty} \frac{3^{k - 1}}{4^k} \sum_{\ell = 0}^{\eta_k} \mathds{1}\Big\{ \tilde{b}_k\Big( \frac{\ell}{\eta_k} \Big) < \frac{1}{2} \Big\} \binom{\eta_k}{\ell} \sum_{q \in \tau^{I_r}} V_r^{I_r}(q) \sum_{i : q_i = 0}  \Big( \prod_{j : q_j = 1} u_s(X_r^j) \Big) \notag \\
&\phantom{= \sum_{k = 1}^{\infty} \Big(\frac{3}{4}\Big)^{k - 1} \frac{1}{4} \sum_{\ell = 0}^{\eta_k} \mathds{1}\{ \tilde{b}_k\}} \times \Big( \prod_{j \neq i : q_j = 0} (1 - u_s(X_r^j)) \Big) u_s(X_r^i)^{\ell} (1 - u_s(X_r^i))^{\eta_k - \ell} \notag \\
&\phantom{=} - I_r \sum_{q \in \tau^{I_r}} V_r^{I_r}(q) \Big( \prod_{j : q_j = 1} u_s(X_r^j) \Big) \Big( \prod_{j : q_j = 0} (1 - u_s(X_r^j)) \Big) \rmd s. \label{B_spde}
\end{align}
Likewise,
\begin{align}
\bar{B}_s(u_r) &= \sum_{k = 1}^{ \infty} \frac{3^{k - 1}}{4^k} \sum_{i = 1}^{I_{s-}} \sum_{q \in \tau^{I_{s-} - 1}} \sum_{\ell = 0}^{\eta_k} \sum_{\tilde{q} \in \tau_{\ell}^{\eta_k}} V_{s-}^{I_{s-}}(M_{i, k} (q \oplus \tilde{q})) (1 - u_r(X_{s-}^i))^{\eta_k - \ell} \notag \\
&\phantom{= \sum_{k = 1}^{ \infty}} \times u_r(X_{s-}^i)^{\ell} \Big( \prod_{j \neq I_{r-} : q_j = 1} u_r(X_{s-}^{j + \mathds{1}\{j \geq i\}}) \Big) \prod_{j \neq I_{s-} : q_j = 0} (1 - u_r(X_{s-}^{j + \mathds{1}\{j \geq i\}})) \notag \notag \\
&\phantom{=} - I_{s-} \sum_{q \in \tau^{I_{s-}}} V_{s-}^{I_{s-}}(q) \Big( \prod_{j : q_j = 1} u_r(X_{s-}^j) \Big) \Big( \prod_{j : q_j = 0} (1 - u_r(X_{s-}^j)) \Big). \label{B_ips}
\end{align}
By \eqref{merger_operator}, for $\tilde{q} \in \tau_{\ell}^{\eta_k}$ we have $V_{s-}^{I_{s-}}(M_{i, k}(q \oplus \tilde{q})) = 1$ exactly when $V_{s-}^{I_{s-}}(q^{1:(i - 1)} \oplus_1 1 \oplus q^{i:I_{s-}}) = 1$ and $\mathds{1}\{\tilde{b}_k(\ell / \eta_k) \geq 1/2\}$, or when $V_{s-}^{I_{s-}}(q^{1:(i - 1)} \oplus_1 0 \oplus q^{i:I_{s-}}) = 1$ and $\mathds{1}\{\tilde{b}_k(\ell / \eta_k) < 1/2\}$.
Otherwise $V_{s-}^{I_{s-}}(M_{i, k}(q \oplus \tilde{q})) = 0$.
Hence the nonzero terms in \eqref{B_spde} and \eqref{B_ips} pair up exactly, and
\begin{equation*}
    \int_0^T \int_0^{T - r} \bbE_{(u_0; x^{1:N}, v^N)}[ B_s^0(X_r^{1:I_r}, V_r^{I_r}) - \bar{B}_s(u_r)] \rmd s \rmd r = 0.
\end{equation*}

The argument for the second term on the right-hand side of \eqref{eps_to_zero} is similar.
Using
\begin{align}
    u_t(x) (1 - u_t(x)) &= u_t(x) - u_t(x)^2, \notag \\
    u_t(x) (1 - u_t(x)) &= (1 - u_t(x)) - (1 - u_t(x))^2, \label{c2}
\end{align}
and \cite[Lemma 2]{athreya/tribe:2000} to justify the $\varepsilon \to 0$ limit, $C_s^0(X_r^{1:I_r}, V_r^{I_r})$ can be written as
\begin{align}
&C_s^0(X_r^{1:I_r}, V_r^{I_r}) \rmd s \notag \\
&= \Bigg( \sum_{q \in \tau^{I_r}} V_r^{I_r}(q) \sum_{\substack{i \neq j :\\ q_i = q_j = 1}} \Big( \prod_{\substack{k \neq i, j :\\ q_k = 1}} u_s(X_r^k) \Big) \Big( \prod_{k : q_k = 0} (1 - u_s(X_r^k) ) \Big) u_s(X_r^i) \notag \\
&\phantom{=} - \sum_{q \in \tau^{I_r}} V_r^{I_r}(q) \sum_{\substack{i \neq j :\\ q_i = q_j = 1}} \Big( \prod_{\substack{k \neq i, j :\\ q_k = 1}} u_s(X_r^k) \Big) \Big( \prod_{k : q_k = 0} (1 - u_s(X_r^k) ) \Big) u_s(X_r^i)^2 \notag \\
&\phantom{=} + \sum_{q \in \tau^{I_r}} V_r^{I_r}(q) \sum_{\substack{i \neq j :\\ q_i = q_j = 0}} \Big( \prod_{k : q_k = 1} u_s(X_r^k) \Big) \Big( \prod_{\substack{k \neq i, j :\\ q_k = 0}} (1 - u_s(X_r^k) ) \Big) (1 - u_s(X_r^i)) \notag \\
&\phantom{=} - \sum_{q \in \tau^{I_r}} V_r^{I_r}(q) \sum_{\substack{i \neq j :\\ q_i = q_j = 0}} \Big( \prod_{k : q_k = 1} u_s(X_r^k) \Big) \Big( \prod_{\substack{k \neq i, j :\\ q_k = 0}} (1 - u_s(X_r^k) ) \Big) (1 - u_s(X_r^i))^2 \notag \\
&\phantom{=} - \sum_{q \in \tau^{I_r}} V_r^{I_r}(q) \sum_{\substack{i \neq j :\\ q_i = 1, q_j = 0}} \Big( \prod_{k : q_k = 1} u_s(X_r^k) \Big) \Big( \prod_{k : q_k = 0} (1 - u_s(X_r^k) ) \Big) \Bigg) \rmd L_s^{i, j} \notag \\
&= \sum_{q \in \tau^{I_r}} V_r^{I_r}(q) \sum_{\substack{i \neq j :\\ q_i = q_j = 1}} \Big( \prod_{\substack{k \neq j :\\ q_k = 1}} u_s(X_r^k) \Big) \Big( \prod_{k : q_k = 0} (1 - u_s(X_r^k) ) \Big) \rmd L_s^{i, j} \notag \\
&\phantom{=} + \sum_{q \in \tau^{I_r}} V_r^{I_r}(q) \sum_{\substack{i \neq j :\\ q_i = q_j = 0}} \Big( \prod_{k : q_k = 1} u_s(X_r^k) \Big) \Big( \prod_{\substack{k \neq j :\\ q_k = 0}} (1 - u_s(X_r^k) ) \Big) \rmd L_s^{i, j} \notag \\
&\phantom{=} -  I_r (I_r - 1) \sum_{q \in \tau^{I_r}} V_r^{I_r}(q) \Big( \prod_{j : q_j = 1} u_s(X_r^j) \Big) \Big( \prod_{j : q_j = 0} (1 - u_s(X_r^j) ) \Big) \rmd L_s^{i, j}, \label{C_spde}
\end{align}
while 
\begin{align}
\rmd \bar{C}_s(u_r) &= \Bigg( \sum_{i = 1}^{I_{s-}} \sum_{j \neq i} \sum_{q \in \tau^{I_{s-} - 1}} V_{s-}^{I_{s-}}(S_{i,j} q) \Big( \prod_{k < I_{r-} - 1 : q_k = 1} u_r(X_{s-}^{k + \mathds{1}\{k \geq i \wedge j\} + \mathds{1}\{k \geq i \vee j\}}) \Big) \notag \\
&\phantom{= \Bigg( \sum_{i = 1}^{I_{s-}} \sum_{j \neq i} \sum_{q \in \tau^{I_{s-} - 1}}} \times \Big( \prod_{k < I_{s-} - 1 : q_k = 0} (1 - u_r(X_{s-}^{k + \mathds{1}\{k \geq i \wedge j\} + \mathds{1}\{k \geq i \vee j\}})) \Big) \notag \\
&\phantom{= \Bigg( \sum_{i = 1}^{I_{s-}} \sum_{j \neq i} \sum_{q \in \tau^{I_{s-} - 1}}} \times u_r(X_{s-}^i)^{\mathds{1}\{q_{I_{s-} - 1} = 1\}} (1 - u_r(X_{s-}^i))^{\mathds{1}\{q_{I_{s-} - 1} = 0\}} \notag \\
&\phantom{= \Bigg(} - I_{s-} (I_{s-} - 1) \sum_{q \in \tau^{I_{s-}}} V_{s-}^{I_{s-}}(q) \Big( \prod_{j : q_j = 1} u_r(X_{s-}^j) \Big)  \Big( \prod_{j : q_j = 0} (1 - u_r(X_{s-}^j)) \Big) \Bigg) \rmd L_s^{i, j}. \label{C_ips}
\end{align}
By \eqref{split_operator}, the sum over $q \in \tau^{I - 1}$ corresponds exactly to summing over those $q \in \tau^I$ in which two elements coincide.
Hence, as before, the terms in \eqref{C_spde} and \eqref{C_ips} pair up exactly, and
\begin{align*}
    &\int_0^T [g(r, 0, 0) - g(0, r, 0)] \rmd r \\
    &= \frac{\nu}{2} \int_0^T \bbE_{(u_0; x^{1:N}, v^N)}\Bigg[ \int_0^{T - r} [ C_s^{\varepsilon}(X_r^{1:I_r}, V_r^{I_r}) \rmd s - \rmd \bar{C}_s(u_r^{\varepsilon})] \Bigg]  \rmd r = 0.
\end{align*}
\end{proof}

\begin{rmk}
It is tempting to repeat the proof of Theorem \ref{thm:bernoulli} for a more general SPDE of the form 
\begin{equation*}
\partial_t u = \frac{1}{2} \Delta u + \mu (\tilde{b}(u) - u) + \sqrt{\nu (\sigma(u) - u^2)} \dot{W},
\end{equation*}
following the coordinated branching used by \cite{athreya/tribe:2000} to construct interacting particle duals for noise coefficients.
However, the analogue of \eqref{c2} required for a dual interpretation akin to ours to hold is
\begin{equation*}
\sigma(u) - u^2 = 1 - \sigma(u) - (1 - u)^2,
\end{equation*}
which necessitates $\sigma(u) = u$.
Thus, it does not seem possible to extend our voting mechanism to equations more general than \eqref{spde}.
\end{rmk}
\section{Proof of Theorem \ref{theo:main theorem} for Baire class-1 drifts}\label{sec:Baire 1}
\begin{defn}
    We recall that Baire class-1 functions are defined to be those functions which are the pointwise limit of a sequence of continuous functions.
\end{defn}

We now weaken the assumption that $\tilde{b}$ in \eqref{spde} is continuous and polynomially bounded, instead assuming only that it's a bounded, measurable and Baire class-1 function $\tilde b:[0,1]\ra [0,1]$.
Our goal is to establish Theorem \ref{theo:main theorem} for such Baire class-1 drifts. 

Since $\tilde{b}:[0,1]\ra [0,1]$ is Baire class-1, we can take a sequence of continuous and polynomially bounded $\tilde b^m:[0,1]\ra [0,1]$ converging pointwise to $b$, and each $\tilde b^m$ induces a dual branching-coalescing voter model.

In the following, $u_t$ is a solution of \eqref{spde} with one of the chosen boundary conditions (i)-(iv) and a fixed drift $b(u)=\mu(\tilde{b}(u)-u)$, and $(X_t^{m,1:I_t(m)}, V_t^{m,I_t(m)})$ is the branching-coalescing particle system with branching mechanism corresponding to $\tilde b^m$.
The constant $\mu>0$ is unchanged.

\begin{theo}\label{thm:bernoulli general drift}
Let $\tilde{b} : [0, 1] \to [0, 1]$ in \eqref{spde} be bounded, measurable and Baire class-1. For $N \in \mathbb{N}$, $x^{1:N} \in [0, 1]^N, v^N \in \{0, 1\}^{2^N}$, and any measurable $u_0 : [0, 1] \to [0, 1]$, we have
\begin{equation}\label{main_eq limiting}
\bbE_{u_0}[ \langle H(u_t; x^{1:N}), v^N \rangle ] = \lim_{m\ra\infty}\bbE_{(x^{1:N}, v^N)}[\langle H(u_0; X_t^{m,1:I_t(m)}), V_t^{m,I_t(m)} \rangle]. 
\end{equation}
\end{theo}

\begin{proof}[Proof of Theorem \ref{thm:bernoulli general drift}]
The proof strategy will be a modification of Theorem \ref{thm:bernoulli}.
We will consider the difference in the two sides of \eqref{main_eq} for fixed $m$, and show it converges to $0$ as $m\ra \infty$.
We emphasise that (a) the solution of the SPDE is for the fixed drift $b(u)=\mu(\tilde{b}(u)-u)$, and (b) this is only a statement about convergence of the value of the duality expressions---we will not use anything about convergence of the SPDE or dual as processes.

Equation \eqref{main_lhs} is unchanged, and we recall it as
\begin{align}
&\bbE_{u_0}[\langle H(u_t^{\varepsilon}; x^{1:N}), v^N \rangle] - \langle H(u_0^{\varepsilon}; x^{1:N}), v^N \rangle \notag \\
&=: \int_0^t \Big( \bbE_{u_0}[A_s^{\varepsilon}(x^{1:N}, v^N)] + \mu \bbE_{u_0}[B_s^{\varepsilon}(x^{1:N}, v^N)] + \frac{\nu}{2} \bbE_{u_0}[C_s^{\varepsilon}(x^{1:N}, v^N)] \Big) \rmd s. \label{eq:recollection expression from SPDE Ito}
\end{align}
We define $B^{m,\epsilon}_s$ to be $B^\epsilon_s$ defined in terms of the solution $u_s$ as before, but with $\tilde{b}$ replaced by $\tilde{b}^m$:
\begin{align*}
&B^{m,\varepsilon}_s := \bbE_{u_0}\Bigg[\sum_{q \in \tau^N} v^N(q) \sum_{i = 1}^N (-1)^{\mathds{1}\{q_i = 0\}} \Big( \prod_{j \neq i : q_j = 1} u_s^{\varepsilon}(x^j) \Big) \Big( \prod_{j \neq i : q_j = 0} (1 - u_s^{\varepsilon}(x^j)) \Big) \\
&\phantom{:= + \mu \bbE\Bigg[\sum_{q \in \tau^N} v^N(q) \sum_{i = 1}^N (-1)^{\mathds{1}\{q_i = 0\}} \Big( \prod_{j \neq i : q_j = 1} u_s^{\varepsilon}(x^j) \Big)} \times ( (\tilde{b}^m \circ u_s)^{\varepsilon}(x^i) - u_s^{\varepsilon}(x^i) ) \Bigg]
\end{align*}
The expressions $A_s^\epsilon$ and $C_s^{\epsilon}$ don't depend on $m$, so are unchanged.
Moreover \eqref{main_rhs} becomes
\begin{align*}
&\bbE_{(x^{1:N}, v^N)}[\langle H(h; X_t^{m,1:I_t(m)}), V_t^{m,I_t(m)} \rangle] - \langle H(h; x^{k,1:N}), v^N \rangle \\
&=: \int_0^t \Big( \bbE_{(x^{1:N}, v^N)}[\bar{A}_s^m(h)] + \mu \bbE_{(x^{1:N}, v^N)}[\bar{B}^m_s(h)] \Big) \rmd s + \frac{\nu}{2} \bbE_{(x^{1:N}, v^N)}\Big[\int_0^t \rmd \bar{C}^m_s(h)\Big],
\end{align*}
where $\bar{A}_s^m$, $\bar{B}^m_s(h)$ and $\bar{C}^m_s(h)$ depend on $m$ via the voter model.

We now allow $g(t,s,\varepsilon)$ to depend on $m$ via the voter model but with $u_t$ fixed as a solution of \eqref{spde} with drift $\tilde{b}$, and write it as $g_m$.
Then \eqref{process_difference} becomes
\begin{align}
    &\int_0^T [g_m(r, 0, \varepsilon) - g_m(0, r, \varepsilon)] \rmd r \notag \\
    &= \int_0^T \int_0^{T - r} \bbE_{(u_0; x^{1:N}, v^N)}[ A_s^{\varepsilon}(X_r^{m, 1:I_r(m)}, V_r^{m, I_r(m)}) - \bar{A}^m_s(u_r^{\varepsilon})] \rmd s \rmd r \notag \\
    &\phantom{=}+ \mu \int_0^T \int_0^{T - r} \bbE_{(u_0; x^{1:N}, v^N)}[ B_s^{m,\varepsilon}(X_r^{m, 1:I_r(m)}, V_r^{m, I_r(m)}) - \bar{B}^m_s(u_r^{\varepsilon})] \rmd s \rmd r \notag \\
    &\phantom{=} + \frac{\nu}{2} \int_0^T \bbE_{(u_0; x^{1:N}, v^N)}\Bigg[\int_0^{T - r} [ C_s^{\varepsilon}(X_r^{m, 1:I_r(m)}, V_r^{m, I_r(m)}) \rmd s - \rmd \bar{C}^m_s(u_r^{\varepsilon})] \Bigg] \rmd r\notag \\
    &\phantom{=} + \mu \int_0^T \int_0^{T - r} \bbE_{(u_0; x^{1:N}, v^N)}[ B_s^{\varepsilon}(X_r^{m, 1:I_r(m)}, V_r^{m, I_r(m)}) \notag \\
    &\phantom{= + \mu \int_0^T \int_0^{T - r} \bbE_{(u_0; x^{1:N}, v^N)}[} - B_s^{m,\varepsilon}(X_r^{m, 1:I_r(m)}, V_r^{m, I_r(m)})] \rmd s \rmd r \label{eq:Baire class-1 appearance of extra term}. 
\end{align}
In \eqref{eq:Baire class-1 appearance of extra term}, the difference between the terms arising out of the drift $\tilde{b}$ of the SPDE \eqref{spde}, and the branching mechanism $\tilde{b}^m$ of the particle system,
\[
\int_0^T \int_0^{T - r} \bbE_{(u_0; x^{1:N}, v^N)}[ B_s^{\varepsilon}(X_r^{m, 1:I_r(m)}, V_r^{m, I_r(m)}) - \bar{B}^m_s(u_r^{\varepsilon})] \rmd s \rmd r,
\]
is then the sum of the second and fourth (last) terms on the right-hand side.

By the continuity of $u$ and $u^{\varepsilon}$ and Lebesgue's differentiation theorem, we have the following almost sure convergence for every $N \in \bbN$, every $q\in \tau^N$, and almost every $(x^1,\ldots,x^N)\in \Rm^N$:
\begin{align}
    &\sum_{q \in \tau^N} v^N(q) \sum_{i = 1}^N (-1)^{\mathds{1}\{q_i = 0\}} \Big( \prod_{j \neq i : q_j = 1} u_s^{\varepsilon}(x^j) \Big)  \Big( \prod_{j \neq i : q_j = 0} (1 - u_s^{\varepsilon}(x^j)) \Big) ( (\tilde{b} \circ u_s)^{\varepsilon}(x^i) - u_s^{\varepsilon}(x^i) ) \notag\\
    &\overset{\varepsilon \to 0}{\ra} \sum_{q \in \tau^N} v^N(q) \sum_{i = 1}^N (-1)^{\mathds{1}\{q_i = 0\}} \Big( \prod_{j \neq i : q_j = 1} u_s(x^j) \Big) \Big( \prod_{j \neq i : q_j = 0} (1 - u_s(x^j)) \Big) \notag \\
    &\phantom{\overset{\varepsilon \to 0}{\ra} \sum_{q \in \tau^N} v^N(q) \sum_{i = 1}^N (-1)^{\mathds{1}\{q_i = 0\}} \Big( \prod_{j \neq i : q_j = 1} u_s(x^j) \Big)} \times ( (\tilde{b} \circ u_s)(x^i) - u_s(x^i) ).\label{almost sure convergence for replacing b}
    \end{align}
This implies, by the bounded convergence theorem, that
\begin{align*}
&\lim_{\varepsilon \to 0}\int_0^T\int_0^{T-r} \bbE_{(x^{1:n}, v^n)}[B_{s}^{\varepsilon}(X_r^{m, 1:I_r(m)}, V_r^{m, I_r(m)})] \rmd s \rmd r \\
&= \int_0^T \int_0^{T - r} \bbE_{(x^{1:n}, v^n)}[B_{s}^{0}(X_r^{m, 1:I_r(m)}, V_r^{m, I_r(m)})] \rmd s \rmd r.
\end{align*}
The same convergence followed in \eqref{eps_to_zero} simply by using the continuity of the drift $b$ to get \eqref{almost sure convergence for replacing b} implicitly; here we need to invoke Lebesgue's differentiation theorem to justify the it for $B^{m, \varepsilon}_s$ in \eqref{almost sure convergence for replacing b}.
We can therefore take $\varepsilon\ra 0$ as in the proof of Theorem \ref{thm:bernoulli} after \eqref{eps_to_zero}, whereupon \eqref{eq:Baire class-1 appearance of extra term} becomes
\begin{align*}
    &\lim_{\varepsilon\ra 0}\int_0^T [g_m(r, 0, \varepsilon) - g_m(0, r, \varepsilon)] \rmd r \\
    &= \mu \int_0^T \int_0^{T - r} \bbE_{(u_0; x^{1:N}, v^N)}[ B_s^0(X_r^{m, 1:I_r(m)}, V_r^{m, I_r(m)})  - B_s^{m, 0}(X_r^{m, 1:I_r(m)}, V_r^{m, I_r(m)})] \rmd s \rmd r.
\end{align*}
From the expression in \eqref{main_lhs}, we have
\begin{align*}
B^0_s (x^{1:N}, v^N) - B_s^{m,0}(x^{1:N}, v^N) = \mu &\sum_{q \in \tau^N} v^N(q) \sum_{i = 1}^N (-1)^{\mathds{1}\{q_i = 0\}} \Big( \prod_{j \neq i : q_j = 1} u_s(x^j) \Big) \\
&\times \Big( \prod_{j \neq i : q_j = 0} (1 - u_s (x^j)) \Big) ( (\tilde{b}-\tilde{b}^m )\circ u_s) (x^i),
\end{align*}
so that the difference on the left-hand side can be bounded above as
\begin{align*}
&\lvert B^0_s (x^{1:N}, v^N) - B_s^{m,0}(x^{1:N}, v^N)\rvert \\
&\leq \mu \sum_{i=1}^N\lvert (\tilde{b}-\tilde{b}^m )\circ u_s) (x^i)\rvert \sum_{q\in \tau^N} \Big( \prod_{j \neq i : q_j = 1} u_s(x^j) \Big) \Big( \prod_{j \neq i : q_j = 0} (1 - u_s (x^j)) \Big) \\
&= 2 \mu \sum_{i=1}^N\lvert (\tilde{b}-\tilde{b}^m )\circ u_s) (x^i)\rvert [u_s(x^j)+1-u_s(x^j)]^{N-1} \\
&= 2 \mu \sum_{i=1}^N\lvert ((\tilde{b}-\tilde{b}^m )\circ u_s) (x^i)\rvert.
\end{align*}
Therefore
\begin{align}
&\Bigg\lvert \int_0^T[g_m(r,0,0)-g_m(0,r,0)]\rmd r \Bigg\rvert\notag \\
&\leq 2\mu\int_0^T\int_0^{T-r}\bbE_{(u_0; x^{1:N}, v^N)}\Bigg[\sum_{i=1}^{I_r(m)}\lvert ((\tilde{b}-\tilde{b}^m )\circ u_s) (X^{m, i})\rvert\Bigg]\rmd s \rmd r. \label{eq:difference of gk in term of duals}
\end{align}
We now let $\tilde{X}^{1:\tilde{I}_t}_t=(\tilde{X}^1_t,\ldots,\tilde{X}_t^{\tilde{I}_t})$ be a branching-coalescing Brownian motion where at rate $\mu$, each particle branches into infinitely many particles, carried on an entirely new probability space $(\tilde {\Omega},\tilde{\mathcal{F}},\tilde{\Pm})$.
This is well-defined by \cite[Theorem 1.4]{barnesetal:2025}.
We write $( \Omega,\mathcal{F},\Pm)$ for our original probability space carrying our solution $u_t$. Then since $\tilde{X}$ stochastically dominates the branching-coalescing Brownian motion in our dual for any $m$, and is independent of $u_t$, we obtain from \eqref{eq:difference of gk in term of duals} that
\begin{align}
  &\Bigg\lvert \int_0^T [g_m(r,0,0)-g_m(0,r,0)]\rmd r\Bigg\rvert\notag \\
  &\leq 2\mu \int_0^T\int_0^{T-r}\bbE_{u_0}\Bigg[\tilde{\bbE}_{(x^{1:N}, v^N)}\Bigg[\sum_{i=1}^{\tilde{I}_r}\lvert ((\tilde{b}-\tilde{b}^m )\circ u_s) (\tilde{X}^i_r)\rvert \Bigg]\Bigg]\rmd s \rmd r. \label{eq:inequality gk bounded infinte bbm}
\end{align}
Now for every $\omega\in \Omega$, $\tilde{\omega}\in \tilde{\Omega}$ and $r,s>0$ with $\tilde{I}_r(\tilde{\omega})<\infty$, 
\begin{equation}\label{eq:sure convergence for extension}
\lim_{m \to \infty} \sum_{i=1}^{\tilde{I}_r(\tilde\omega)}\lvert ((\tilde{b}-\tilde{b}^m )\circ u_s(\omega)) (\tilde{X}^i_r(\tilde\omega))\rvert = 0
\end{equation}
as $m\ra\infty$, since $\tilde{b}^m \ra \tilde{b}$ pointwise.
The left-hand side is also clearly bounded by $\tilde{I}_r$ since $\tilde{b}$ and $\tilde{b}^m$ both take values in $[0,1]$.
To apply the dominated convergence theorem to see that the right-hand side of \eqref{eq:inequality gk bounded infinte bbm} converges to $0$ as $m\ra\infty$, all that remains is to check that
\[
\int_0^T \tilde{\expE}_{(x^{1:N}, v^N)}[\tilde{I}_r ] \rmd r < \infty.
\]
This follows from \cite[Theorem 1.4]{barnesetal:2025} and Tonelli's theorem.
\end{proof}

\subsection*{Existence of weak solutions}

We will firstly need the following tightness lemma, which extends \cite[Lemma A.6]{Fan2023} to bounded, measurable drifts.
This section is concerned with Baire class-1 drifts, but we will implicitly use the following lemma in Section \ref{sec:measurable} when we deal with drifts which are only measurable.
Hence we will establish Lemma \ref{lem:tightness} assuming only boundedness and measurability.
Whilst we are confident it was previously known, we couldn't find a reference; \cite[Lemma A.6]{Fan2023} takes the drift to be of the form $\beta u(1-u)$ for some $\beta\in\Rm$.
Our proof amounts to simply checking that this is not important. 

In the following, $C_{b,\loc}(\Rm_{>0}\times E;[0,1])$ is the space of \textit{globally} bounded and continuous functions $\Rm_{>0}\times E\ra [0,1]$, equipped with the topology of \textit{locally} uniform convergence on compact sets.

\begin{lem}\label{lem:tightness}
Let $\nu>0$ and $C<\infty$ be fixed. Then the set of all solutions $u_t(x)$ to \eqref{spde}, for all possible initial conditions and all measurable drifts $b:E\times \Rm\ra [0,1]$ with $\| b\|_{\infty}\leq C$, is tight in $C_{b,\loc}(\Rm_{>0}\times E;[0,1])$.
\end{lem}
\begin{proof}[Proof of Lemma \ref{lem:tightness}]
    We let $u$ be an arbitrary solution to \eqref{spde} for some admissible drift $b$, and let $K\subseteq E\times [0,1]$ be compact.
    Then there exists $C<\infty$, which may depend upon $K$ but not on $u$ or $b$, such that for all $(t_1,x_1),(t_2,x_2)\in K$,
    \begin{equation}\label{eq:Holder estimate}
        \expE_{u_0}[\lvert u_{t_1}(x_1)-u_{t_2}(x_2)\rvert^p]\leq C(\lvert t_1-t_2\rvert^{p/4}+\lvert x_1-x_2\rvert^{p/2}).
    \end{equation}
This can be found e.g.\ in \cite[Lemma 4]{Fan2021}, for regular drift, but we couldn't find a statement for general bounded drifts. Nevertheless, allowing $b$ to be a general bounded drift causes no additional problems.
In particular, to obtain \eqref{eq:Holder estimate}, assume $t_2\geq t_1$ and use \eqref{E:MildSol_FKPP} to see that
\begin{align*}
&u_{t_2}(x_2)-u_{t_1}(x_1) \\
&= \int_{E} \{[p(t_2,y,x_2)-p(t_2,y,x_1)]+[p(t_2,y,x_1)-p(t_1,y,x_1)]\}\,u_0(y)\rmd y \\
&\phantom{=} + \int_0^{t_1}\int_{E}\{[p(t_2-s,y,x_2)-p(t_2-s,y,x_1)]+[p(t_2-s,y,x_1)-p(t_1-s,y,x_1)]\}  b(u_s(y))\rmd y\rmd s \\
&\phantom{=} + \int_{t_1}^{t_2}p(t_2-s,y,x_2)b(u_s(y))\rmd y\rmd s \\
&\phantom{=}+ \int_{E\times [0,t_1]}\{[p(t_2-s,y,x_2)-p(t_2-s,y,x_1)]+[p(t_2-s,y,x_1)-p(t_1-s,y,x_1)]\}\\
&\phantom{= + \int_{E\times [0,t_1]}} \times \sqrt{\nu u_s(y) (1-u_s(y))}\rmd W(y,s)\\
&\phantom{=} + \int_{E\times [0,t_1]}(t_2,y,x_2)\sqrt{\nu u_s(y) (1-u_s(y))}\rmd W(y,s).
\end{align*}
Then \eqref{eq:Holder estimate} follows from heat kernel estimates, and the BDG inequality applied to the last two terms. Using this, Lemma \ref{lem:tightness} follows as in the proof of \cite[Lemma A.6]{Fan2023}.\end{proof}

As in the proof of Theorem \ref{thm:bernoulli general drift}, we assume that $\tilde{b}:[0,1]\ra [0,1]$ is Baire class-1, and take a sequence of continuous and polynomially bounded functions $\tilde{b}^m:[0,1]\ra [0,1]$ converging pointwise to $\tilde{b}$.
We write $b(u)=\mu (\tilde{b}(u)-u)$ and similarly $b^m(u)=\mu (\tilde{b}^m(u)-u)$.
The existence of weak solutions to \eqref{eq:SPDE 1} for continuous drifts follows from \cite[Theorem 2.6]{shiga:1994}, hence for each $b^m$.
Using this, we take a sequence of solutions to \eqref{spde} for each $m$, $(u^m:m\in \mathbb{N})$.
The key difficulty is that we have no guarantee that $b^m(u_t^m(x))$ converges to $b(u_t(x))$, due to the discontinuity of $b$.

By Lemma \ref{lem:tightness}, the $u^m$s are tight in the space of locally uniform convergence on $\Rm_{>0}\times E$. 
Meanwhile, identifying $L^{\infty}(\Rm_{\geq 0}\times E)$ with measures having $L^{\infty}$ densities, the $b^m(u^m(t,x))$s are tight in the topology of weak convergence of measures, with all subsequential limits supported on $L^{\infty}(\Rm_{\geq 0}\times E;[-\mu,\mu])$.
Finally, taking a countable dense family $(\psi_\ell:1\leq \ell<\infty)$ in $C^{\infty}_c(\Rm_{\geq 0}\times E)$, we see that the family 
\begin{equation}\label{eq:family of martingale measures}
M^{m,\psi_\ell}_t:=\int_0^t\int_E\psi_\ell(y,s)\sqrt{\nu u^m_s(y)(1-u^m_s(y))}\rmd W^m(y,s)
\end{equation}
are a countable collection of martingale measures with quadratic covariation 
\begin{equation}\label{eq:family of martingale measures QV}
\langle M^{m,\psi_{\ell_1}},M^{m,\psi_{\ell_2}}\rangle_t=\int_0^t \int_E \psi_{\ell_1}(y)\psi_{\ell_2}(y) \nu u^m_s(y)(1-u^m_s(y)) \rmd y \rmd s.
\end{equation}
By Aldous' criterion, it follows that $\{M^{m,\psi_{\ell}}_t:1\leq m<\infty\}$ is tight in $C([0,\infty);\Rm)$ (equipped with a metrisation of uniform convergence on compact time-intervals), for each $1\leq \ell <\infty$. We use tightness, pass to a subsequence (which we don't denote to keep notation cleaner) and use Skorokhod's representation theorem to put all of these on a common probability space on which we have convergence of all of the above in probability. We denote the resultant subsequential limits as $u_t$, $\hat b(t,x)$, and $M^{\psi_{\ell}}_t$ for $1\leq \ell <\infty$, respectively. 

Now we see that each $M^{\psi_{\ell}}_t$ is again a martingale, with quadratic covariation 
\[
\langle M^{\psi_{\ell_1}},M^{\psi_{\ell_2}}\rangle_t=\int_0^t \int_E \psi_{\ell_1}(y)\psi_{\ell_2}(y) \nu u_s(y)(1-u_s(y))\rmd y \rmd s.
\]
We then construct the orthogonal martingale measure $M(s,y)$ via duality, and define
\[
W(y,s):=\frac{\Ind(0<u_s(y)<1)}{\sqrt{\nu u_s(y)(1-u_s(y))}}M(s,y)+\Ind(u_s(y)\in \{0,1\})\hat{W}(y,s),
\]
where $\hat{W}$ is a separate, exogenous white noise, so that $W$ is also a white noise.

We obtain that $u$ satisfies
\begin{align}
\notag  u_t(x) = {}& \int_{E} p(t,y,x)\,u_0(y)\rmd y + \int_0^t\int_Ep(t-s,y,x)\hat{b}(s,y)\rmd y\rmd s  \\&+ \int_{E\times [0,t]}p(t-s,y,x)\,
\sqrt{\nu u_s(y) (1-u_s(y))}\,\rmd W(y,s), \label{eq:martingale for existence using general B}
\end{align}
for the adapted, bounded, measurable $\hat{b}:\Rm_{\geq 0}\times E\ra \Rm$ which we previously obtained as a subsequential limit. 

We want to replace $\hat{b}(s,y)$ with $b(u_s(y))$.
To this end, our overall goal is to show that
\begin{equation}\label{eq:Jt diff mean 0}
J_t(x):=\int_0^t\int_Ep(t-s,y,x)[\hat{b}(s,y)-b(u_s(y))]\rmd y \rmd s =0
\end{equation}
for all $t>0$ and $x\in E$, almost surely.
The key proposition we need to prove this is the following.
\begin{prop}\label{prop:conditional expectation 0 for changes in Jt}
For any $0\leq u<t$, and with $(\mathcal{F}_t:t\geq 0)$ being the filtration on which the SPDE is defined, we have
\begin{equation}\label{eq:conditional expectation 0 for changes in Jt}
\expE\Bigg[\int_u^t\int_Ep(t-s,y,x)[\hat{b}(s,y)-b(u(s,y))]\rmd y\rmd s\Bigg\lvert \mathcal{F}_u\Bigg]=0.
\end{equation}
\end{prop}
Once we have proven Proposition \ref{prop:conditional expectation 0 for changes in Jt}, we will obtain \eqref{eq:Jt diff mean 0} by a martingale argument.
\begin{proof}[Proof of Proposition \ref{prop:conditional expectation 0 for changes in Jt}]
Our strategy is to establish, via duality, that $u_t$ satisfies \eqref{eq:recollection expression from SPDE Ito} in a limiting sense as $\varepsilon\ra 0$, when $B_s^{\varepsilon}$ is defined using $\tilde{b}$. To this end, define the difference $D_t^{\varepsilon}$ by
\begin{align}
D_t^\varepsilon(x^{1:N},v^N):= {}& \bbE_{u_0}[\langle H(u_t^{\varepsilon}; x^{1:N}), v^N \rangle ]- \langle H(u_0^{\varepsilon}; x^{1:N}), v^N \rangle  \notag\\
&- \int_0^t\bbE_{u_0}\Big[ A_s^{\varepsilon}(x^{1:N}, v^N) + \mu B_s^{\varepsilon}(x^{1:N}, v^N) + \frac{\nu}{2}C_s^{\varepsilon}(x^{1:N}, v^N) \Big] \rmd s,\label{eq:Dt diff definition}
\end{align}
for $v^N\in \tau^N$ and $\epsilon> 0$. Here $A^\epsilon,$ $B^{\epsilon}$ and $C^{\epsilon}$ are defined as in \eqref{main_lhs}, using the subsequential limit $u_t$ and drift $\tilde b$. We will firstly need to bound $D_t^\epsilon$.
\begin{lem}\label{lem:bound Dt epsilon}
We have $\lvert D^{\epsilon}_t(x^{1:N},v^N)\rvert\leq 2\mu Nt$ for all $x^{1:N}\in E^N,v^N\in\{0,1\}^N,\epsilon>0, N\in\mathbb{N}$. 
\end{lem}
\begin{proof}[Proof of Lemma \ref{lem:bound Dt epsilon}]
Our starting point is the sequence of solutions to \eqref{spde} with drift $b^m$, $(u^m:m\in \mathbb{N})$.
We restrict our attention to a subsequence along which $u^m$ converges to $u$ as above.
We then define ${\bf{A}}^{m,\epsilon}$, ${\bf{B}}^{m,\epsilon}$ and ${\bf{C}}^{m,\epsilon}$ as $A^{\epsilon}$, $B^{\epsilon}$ and $C^{\epsilon}$ are defined in \eqref{main_lhs} using $u^m$ in place of $u$, and $\tilde{b}^m$ in place of $\tilde{b}$ for ${\bf{B}}^{m,\epsilon}$. It then follows from \eqref{main_lhs} that
\begin{align*}
&\Bigg\lvert\bbE_{(u_0; x^{1:N}, v^N)}[\langle H(u_t^{m,\varepsilon}; x^{1:N}), v^N \rangle] - \langle H(u_0^{\varepsilon}; x^{1:N}), v^N \rangle \\
&\phantom{=}- \int_0^t \bbE_{u_0}\Big[ {\bf A}_s^{m,\varepsilon}(x^{1:N}, v^N)+ \frac{\nu}{2}{\bf C}_s^{m,\varepsilon}(x^{1:N}, v^N) \Big] \rmd s\Bigg\rvert \\
&=\Bigg\lvert\int_0^t  \bbE_{u_0}\Big[\mu {\bf B}_s^{m,\varepsilon}(x^{1:N}, v^N) \Big]  \rmd s\Bigg\rvert\leq \mu Nt,
\end{align*}
where in the last inequality we are using the observation that ${\bf B}^{
m,\epsilon}(x^{1:N},v^N)\leq N$.
Passing to the limit as $m\ra\infty$ and using $B^{\epsilon}(x^{1:N},v^N)\leq N$, we similarly obtain
\begin{align*}
\lvert D_t^\varepsilon(x^{1:N},v^N)\rvert &\leq \Big\lvert \bbE_{u_0}[\langle H(u_t^{\varepsilon}; x^{1:N}), v^N \rangle ]- \langle H(u_0^{\varepsilon}; x^{1:N}), v^N \rangle  \notag\\
&\phantom{=}- \int_0^t\bbE_{u_0}\Big[ A_s^{\varepsilon}(x^{1:N}, v^N)  + \frac{\nu}{2}C_s^{\varepsilon}(x^{1:N}, v^N) \Big] \rmd s\Big\rvert + \Big\lvert \int_0^t\mu B_s^{\varepsilon}(x^{1:N}, v^N)ds \Big\rvert\\
&\phantom{=}\leq \mu N t+\mu N t=2\mu N t,
\end{align*}
whence we obtain Lemma \ref{lem:bound Dt epsilon}.
\end{proof}

For $N=1$, we define $v_0$ by $v_0(0)=0$ and $v_0(1)=1$.
We will view $D^{\epsilon}_t(x,v_0)$ as a function of $(t,x)$, which we emphasize by writing $D^{\epsilon}(t,x) := D_t^{\epsilon}(x,v_0)$.
By Lemma \ref{lem:bound Dt epsilon}, $D^{\epsilon}(t,x)$ is bounded by $2\mu t$.
For locally bounded functions we say $f_\epsilon(t,x)\ra f(t,x)$ vaguely if we have vague convergence of the measures with Lebesgue densities $f_\epsilon$ and $f$, respectively, as $\epsilon\ra 0$, i.e. 
\begin{equation*}
\int_{\Rm_{\geq 0}\times E}f_{\epsilon}(t,x)\phi(t,x)\rmd y\rmd x\ra \int_{\Rm_{\geq 0}\times E}f(t,x)\phi(t,x)\rmd y\rmd x
\end{equation*}
as $\epsilon\ra 0$, for all $\phi\in C_c(\Rm_{\geq 0}\times E)$. 
\begin{prop}\label{prop:D0=0}
We have $D^\epsilon(t,x) \ra 0$ vaguely as $\epsilon\ra 0$.
\end{prop}

\begin{proof}[Proof of Proposition \ref{prop:D0=0}]
We follow the proof of Theorem \ref{thm:bernoulli general drift}, allowing the voter model to be parametrised by $m$ as before but defining $u_t$ to be our subsequential limit introduced in the paragraph below \eqref{eq:family of martingale measures QV}.
Our goal is to obtain an analogue of \eqref{eq:Baire class-1 appearance of extra term}, whose proof adapts that of \eqref{process_difference}. In particular, \eqref{eq:g decomposition} becomes
\begin{align}
    &\int_0^T [g_m(r, 0, \varepsilon) - g_m(0, r, \varepsilon)] \rmd r \notag\\
    &= \int_0^T [g_m(T - r, r, \varepsilon) - g_m(0, r, \varepsilon)] \rmd r - \int_0^T [g_m(r, T - r, \varepsilon) - g_m,(r, 0, \varepsilon)] \rmd r.\label{eq:gm decomp}
\end{align}

We recall again (see Remark \ref{rmk:where b enters proof from SPDE}) that \eqref{main_lhs}, also recalled in \eqref{eq:recollection expression from SPDE Ito}, was the one place where $\tilde b$ enters the proof of Theorem \ref{thm:bernoulli} by virtue of being the drift term in \eqref{spde}. Here \eqref{main_lhs} is replaced by \eqref{eq:Dt diff definition}, which we substitute into the first term on the right hand side of \eqref{eq:gm decomp} to obtain
\begin{align*}
&\int_0^T [g_m(T - r, r, \varepsilon) - g_m(0, r, \varepsilon)] \rmd r \\
&=\int_0^T\bbE_{(u_0;x^{1:N},v^N)}[\langle H(u_{T-r}^{\varepsilon}; X_r^{m,1:I_r(m)}), V_r^{m,I_r(m)} \rangle - \langle H(u_0^{\varepsilon}; X_r^{m,1:I_r(m)}), V_r^{m,I_r(m)} \rangle]\rmd r \\
&=\int_0^T \int_0^{T - r} \bbE_{(u_0; x^{1:N}, v^N)}\Big[ A_s^{\varepsilon}(X_r^{m, 1:I_r(m)}, V_r^{m, I_r(m)}) + \mu B_s^{\varepsilon}(X_r^{m, 1:I_r(m)}, V_r^{m, I_r(m)}) \\
&\phantom{= \int_0^T \int_0^{T - r} \bbE_{(u_0; x^{1:N}, v^N)}[} + \frac{\nu}{2} C_s^{\varepsilon}(X_r^{m, 1:I_r(m)}, V_r^{m, I_r(m)})\Big] \rmd s \rmd r \\
&\phantom{=} + \int_0^T\bbE_{(x^{1:N},v^N)}[D_{T-r}^{\epsilon}(X_r^{m, 1:I_r(m)}, V_r^{m, I_r(m)})]\rmd r.
\end{align*}
Note that in the definition of $D_s^{\epsilon}$, we held the argument $(x^{1:N},v^N)$ fixed while taking an expectation over $u_t$.
Hence $D_s^{\epsilon}$ will only ever be a function of the voter model, so we only need to specify the initial condition of the voter model in the corresponding expectations.  

In the following $B^{m,\epsilon}$ depends on $m$ via $\tilde{b}^m$, but not on $u_t$, which is our subsequential limit. It is therefore different to ${\bf B}^{m,\epsilon}$ in the proof of Lemma \ref{lem:bound Dt epsilon}. We get that \eqref{main_rhs} is unchanged, except that it now depends on $m$ via the voter model as in \eqref{eq:Baire class-1 appearance of extra term}. Substituting this into the second term on the right-hand side of \eqref{eq:gm decomp} as before, and both adding and subtracting $B^{m,\epsilon}_s$, we obtain the following analogue of \eqref{eq:gm decomp}:
\begin{align*}
    &\int_0^T [g_m(r, 0, \varepsilon) - g_m(0, r, \varepsilon)] \rmd r \notag \\
    &= \int_0^T \int_0^{T - r} \bbE_{(u_0; x^{1:N}, v^N)}[ A_s^{\varepsilon}(X_r^{m, 1:I_r(m)}, V_r^{m, I_r(m)}) - \bar{A}^m_s(u_r^{\varepsilon})] \rmd s \rmd r  \\
    &\phantom{=} + \mu \int_0^T \int_0^{T - r}\bbE_{(u_0; x^{1:N}, v^N)} [ B_s^{m,\varepsilon}(X_r^{m, 1:I_r(m)}, V_r^{m, I_r(m)}) - \bar{B}^m_s(u_r^{\varepsilon})] \rmd s \rmd r  \\
    &\phantom{=} + \mu \int_0^T \int_0^{T - r}\bbE_{(u_0; x^{1:N}, v^N)} [ B_s^{\varepsilon}(X_r^{m, 1:I_r(m)}, V_r^{m, I_r(m)})  - B_s^{m,\varepsilon}(X_r^{m, 1:I_r(m)}, V_r^{m, I_r(m)})] \rmd s \rmd r \\
    &\phantom{=} + \frac{\nu}{2} \int_0^T \bbE_{(u_0; x^{1:N}, v^N)}\Big[\int_0^{T - r} [ C_s^{\varepsilon}(X_r^{m, 1:I_r(m)}, V_r^{m, I_r(m)}) \rmd s - \rmd \bar{C}^m_s(u_r^{\varepsilon})] \Big]\rmd r  \\
    &\phantom{=}+\int_0^T\bbE_{( x^{1:N}, v^N)}
    [D_{T-r}^{\epsilon}(X_r^{m, 1:I_r(m)}, V_r^{m, I_r(m)})] \rmd r.
\end{align*}
We see that the last term represents the difference with \eqref{eq:gm decomp}. As we take $\varepsilon\ra 0$ and then $m\ra \infty$, all except the last term on the right-hand side converge to $0$ by following the proof of duality for Baire class-1 drifts, and the left-hand side converges to $0$ by convergence in distribution of $u^m$ to $u$ (recall the paragraph after \eqref{eq:family of martingale measures QV}).
This implies 
\[
\lim_{m\ra\infty}\lim_{\epsilon\ra 0}\bbE_{(x^{1:N}, v^N)}\Big[\int_0^TD_{T-r}^{\epsilon}(X_r^{m, 1:I_r(m)}, V_r^{m, I_r(m)})\rmd r\Big]= 0.
\]

We now fix $\delta>0$.
In the above argument, we can (i) replace $u_0$ with the random initial condition $u_s$ in the definition of $D^{\epsilon}_t$ in \eqref{eq:Dt diff definition}, with the time interval now of length $h$ rather than $T$, (ii) take $N=1$, and (iii) replace $x^{1:N}$ with $\rho_{x,\delta}$, the distribution obtained by evolving a single Brownian motion from $x\in E$ for time $\delta$ (i.e.\ $\text{Normal}(x,\delta)$ if $E=\Rm$).
The reason for this choice is found later in \eqref{eq:step where we use special ic}.
Then, due to (i), $D_{T-r}^{\epsilon}(\cdot)$ becomes $D_{s+h-r}^{\epsilon}(\cdot)-D_{s}^{\epsilon}(\cdot)$, and for any $s>0$ we get
\begin{align}
&\lim_{m\ra\infty}\lim_{\epsilon\ra 0}\bbE_{(\rho_{x,\delta}, v^N)}\Big[\int_0^h [D_{s+h-r}^{\epsilon}(X_r^{m, 1:I_r(m)}, V_r^{m, I_r(m)})\notag\\
&\phantom{\lim_{m\ra\infty}\lim_{\epsilon\ra 0}\bbE_{(\rho_{x,\delta}, v^N)}\Big[\int_0^h [} - D_{s}^{\epsilon}(X_r^{m, 1:I_r(m)}, V_r^{m, I_r(m)})] \rmd r\Big] = 0.\label{eq:difference of Ds fixed t}
\end{align}

In the following, we will use Fubini's theorem and the dominated convergence theorem throughout, justified by Lemma \ref{lem:bound Dt epsilon} along with \cite[Theorem 1.4]{barnesetal:2025}.
We integrate \eqref{eq:difference of Ds fixed t} over $s$ and use two changes of variable, first $u := s + h - r$ and then $s = u$ to obtain
\begin{align}
&\lim_{m\ra\infty}\lim_{\epsilon\ra 0}\bbE_{(\rho_{x,\delta}, v_0)}\Big[\int_0^t\int_0^hD_{s+h-r}^{\epsilon}(X_r^{m, 1:I_r(m)}, V_r^{m, I_r(m)})\rmd r\rmd s \notag\\
&\phantom{\lim_{m\ra\infty}\lim_{\epsilon\ra 0}\bbE_{(\rho_{x,\delta}, v_0)}\Big[}-\int_0^t\int_0^hD_{s}^{\epsilon}(X_r^{m, 1:I_r(m)}, V_r^{m, I_r(m)}) \rmd r\rmd s\Big]\notag\\
&=\lim_{m\ra\infty}\lim_{\epsilon\ra 0}\bbE_{(\rho_{x,\delta}, v_0)}\Big[\int_0^h\int_{h-r}^{t+h-r}D_{u}^{\epsilon}(X_r^{m, 1:I_r(m)}, V_r^{m, I_r(m)})\rmd u\rmd r \notag\\
&\phantom{=\lim_{m\ra\infty}\lim_{\epsilon\ra 0}\bbE_{(\rho_{x,\delta}, v_0)}\Big[}-\int_0^h\int_0^tD_{s}^{\epsilon}(X_r^{m, 1:I_r(m)}, V_r^{m, I_r(m)})\rmd s \rmd r\Big]\notag\\
&=\lim_{m\ra\infty}\lim_{\epsilon\ra 0}\bbE_{(\rho_{x,\delta}, v_0)}\Big[\int_0^h\int_t^{t+h-r}D_{s}^{\epsilon}(X_r^{m, 1:I_r(m)}, V_r^{m, I_r(m)})\rmd s\rmd r \notag\\
&\phantom{=\lim_{m\ra\infty}\lim_{\epsilon\ra 0}\bbE_{(\rho_{x,\delta}, v_0)}\Big[}-\int_0^h\int_0^{h-r}D_{s}^{\epsilon}(X_r^{m, 1:I_r(m)}, V_r^{m, I_r(m)})\rmd s \rmd r\Big] = 0.\label{eq:integral manipulation D final term}
\end{align}
It follows from Lemma \ref{lem:bound Dt epsilon} and \cite[Theorem 1.4]{barnesetal:2025} that the last double integral in \eqref{eq:integral manipulation D final term} is bounded by $Ch^3$ for some $C<\infty$ independent of $m$ and $\epsilon$.
Hence the last line gives
\begin{equation}\label{eq:first integral D neighbourhood normalised by h squared}
\lim_{h\ra 0}\lim_{m\ra\infty}\lim_{\epsilon\ra 0}\frac{1}{h^2}\int_0^h\int_t^{t+h-r}\bbE_{(\rho_{x,\delta}, v_0)}[D_{s}^{\epsilon}(X_r^{m, 1:I_r(m)}, V_r^{m, I_r(m)})]\rmd s\rmd r =0.
\end{equation}
If we evolve the dual for time $r$ started with a single particle, then with probability $e^{-\mu r}$ there are no branches (and no coalescence events since we started from a single particle), so the particle has just moved as a Brownian motion.
This motivates the initial condition in the first term on the right-hand side of the following equation, which was why we defined $\rho_{x,\delta}$ the way we did.
We let $F_{r}$ be the event of the dual branching in time $r$. Then we have
\begin{align}
&\bbE_{(\rho_{x,\delta}, v_0)}[D_{s}^{\epsilon}(X_r^{m, 1:I_r(m)}, V_r^{m, I_r(m)})]\notag \\
&=e^{-\mu r}\bbE_{X\sim \rho_{x,\delta+r}}[D_{s}^{\epsilon}(X, v_0)]+\bbE_{(\rho_{x,\delta}, v_0)}[D_{s}^{\epsilon}(X_r^{m, 1:I_r(m)}, V_r^{m, I_r(m)})\Ind_{F_r}].\label{eq:step where we use special ic} 
\end{align}
We substitute \eqref{eq:step where we use special ic} into \eqref{eq:first integral D neighbourhood normalised by h squared} to obtain
\begin{align}
    0 = {}& \lim_{h\ra 0}\lim_{\epsilon\ra 0}\frac{1}{h^2}\int_0^h\int_t^{t+h-r}\bbE_{X\sim \rho_{x,\delta}}[D_{s}^{\epsilon}(X, v_0)]\rmd s\rmd r \notag \\
    &+ \lim_{h\ra 0}\lim_{\epsilon\ra 0}\frac{1}{h^2}\int_0^h\int_t^{t+h-r}(\bbE_{X\sim \rho_{x,\delta+r}}[D_{s}^{\epsilon}(X, v_0)]-\bbE_{X\sim \rho_{x,\delta}}[D_{s}^{\epsilon}(X, v_0)])\rmd s \rmd r \notag\\
    &+ \lim_{h\ra 0}\lim_{\epsilon\ra 0}\frac{1}{h^2}\int_0^h\int_t^{t+h-r}(e^{-\mu r}-1)\bbE_{X\sim \rho_{x,\delta+r}}[D_{s}^{\epsilon}(X, v_0)]\rmd s \rmd r \notag\\
    &+ \lim_{h\ra 0}\lim_{m\ra\infty}\lim_{\epsilon\ra 0}\frac{1}{h^2}\int_0^h\int_t^{t+h-r}\bbE_{(\rho_{x,\delta}, v_0)}[D_{s}^{\epsilon}(X_r^{m, 1:I_r(m)}, V_r^{m, I_r(m)})\Ind_{F_r}]\rmd s\rmd r, \label{eq:vague_convergence_expansion}
\end{align}
we have removed the $\lim_{m\ra\infty}$ in all terms where there is no dependence on $m$.
Our aim is to show that the second, third, and fourth terms on the right-hand side vanish, so that the first must as well.

For the second and third terms, Lemma \ref{lem:bound Dt epsilon} yields the respective bounds
\begin{align}
    |(\bbE_{X\sim \rho_{x,\delta+r}}[D_{s}^{\epsilon}(X, v_0)]-\bbE_{X\sim \rho_{x,\delta}}[D_{s}^{\epsilon}(X, v_0)])| &\leq \int_E |D^\epsilon_s(y,v_0)| |\rho_{x,\delta+r}(y)-\rho_{x,\delta}(y)| \rmd y \notag \\
    &\leq 2\mu s \lvert\lvert \rho_{x,\delta+r}(\cdot)-\rho_{x,\delta}(\cdot)\rvert\rvert_{L^1(E)}, \label{eq:2nd_term_ub} \\
    |(e^{-\mu r}-1)\bbE_{X\sim \rho_{x,\delta+r}}[D_{s}^{\epsilon}(X, v_0)]| &\leq 2 \mu^2 s h, \label{eq:3rd_term_ub}
\end{align}
both of which vanish as $r \leq h \to 0$.
To control the fourth term, we write $\tau$ for the first time at which the particle branches. Using Tonelli's theorem, extending the domain of integration, using Lemma \ref{lem:bound Dt epsilon}, and finally the change of variable $r = \tau + u$, the absolute value of the corresponding double integral is bounded by
\begin{align*}
&\int_t^{t+h}\int_0^h\bbE_{(\rho_{x,\delta}, v_0)}[\lvert D_{s}^{\epsilon}(X_r^{m, 1:I_r(m)}, V_r^{m, I_r(m)})\rvert\Ind(\tau<r)]\rmd r\rmd s \\
&\leq 2\mu \int_t^{t+h}\int_0^hs\bbE_{(\rho_{x,\delta}, v_0)}[N_r\Ind(\tau<r)]\rmd r\rmd s\\
&\leq 2\mu (t+h)h\bbE_{(\rho_{x,\delta}, v_0)}\Big[\Ind(\tau<h)\expE\Big[\int_{0}^hN_{\tau+u}\rmd u\Big\lvert \mathcal{F}_{\tau}\Big]\Big].
\end{align*}
Given Remark \ref{rmk:applying Barnes to other bcs} and writing $N^0$ for the population of a branching-coalescing Brownian motion on $\Rm$ starting with infinitely many particles at $0$, and with all branching events resulting in infinitely many child particles, we have the bound
\[
\expE\Big[\int_{0}^hN_{\tau+u}\Big\lvert \mathcal{F}_{\tau+u}\Big]\leq \expE\Big[\int_0^hN^0_udu\Big]\ra 0
\]
as $h\ra 0$ by the dominated convergence theorem; $\int_0^hN^0_u \rmd u$ is dominated by $\int_0^1N^0_u \rmd u$ for $h<1$, which is integrable by \cite[Theorem 1.4]{barnesetal:2025}, and the former converges to $0$ almost surely as $h\ra 0$.
By substituting this, as well as \eqref{eq:2nd_term_ub} and \eqref{eq:3rd_term_ub} into \eqref{eq:vague_convergence_expansion}, and using $\Pm(\tau<h)\leq \mu h$, we obtain
\begin{align}
&\Bigg\lvert \lim_{h\ra 0}\lim_{\epsilon\ra 0}\frac{1}{h^2}\int_0^h\int_t^{t+h-r}\Big\{\int_{E}D^{\epsilon}(s,y)\rho_{x,\delta}(\rmd y)\Big\}\rmd s\rmd r\Bigg\rvert\notag\\
&\leq \lim_{h\ra 0}\Bigg[\frac{2\mu (t+h)h\Pm(\tau<h)}{h^2}\expE\Big[\int_0^hN^0_udu\Big]\Bigg]=0.\label{eq:expression for integral Depsilon number particles}
\end{align}

We now use \eqref{eq:expression for integral Depsilon number particles} to prove vague convergence of $D^\epsilon$ to $0$.
We write $p(t,x,y)$ for the heat kernel and use Lemma \ref{lem:bound Dt epsilon} to apply Fubini's theorem.
Given $\phi\in C_c(\Rm_{\geq 0}\times E)$, we have
\begin{align*}
&\Bigg\lvert\int_0^\infty\int_E\phi(t,x) \frac{1}{h^2}\int_0^h\int_t^{t+h-r}\Big\{\int_{E}D^{\epsilon}(s,y)\rho_{x,\delta}(\rmd y)\Big\}\rmd s\rmd r \rmd x \rmd t-\frac{1}{2}\int_0^{\infty}\int_E \phi(s,y)D^\epsilon(s,y)\rmd s\rmd y\Bigg\rvert\notag\\
&=\Bigg\lvert\int_0^{\infty}\Bigg[\frac{1}{h^2}\int_0^h\int_E\int_{\{t > 0: t\leq s\leq t+h-r\}}\phi(t,x)p(\delta,x,y)\rmd t\rmd r \rmd x-\frac{\phi(s,y)}{2}\Bigg]D^{\epsilon}(s,y)\rmd s\rmd y\Bigg\rvert\notag\\
&\leq 2\mu\int_E\int_0^{\infty}s\frac{\phi(s,y)}{h^2}\int_0^h(h-r)\Bigg\lvert \int_Ep(\delta,x,y)\rmd x -1\Bigg\rvert\rmd r\rmd s\rmd y\notag\\
&\phantom{=}+2\mu \int_E\int_0^{\infty}\frac{s}{h^2}\int_0^h\int_E\int_{s+r-h}^s\lvert \phi(t,x)-\phi(s,y)\rvert p(\delta,x,y)\rmd t\rmd r \rmd x\rmd s\rmd y\notag,
\end{align*}
where the inequality follows by adding and subtracting $\phi(s, y)$, as well as Lemma \ref{lem:bound Dt epsilon}.
Now let $C<\infty$ be a constant depending only on $\lvert\lvert \phi\rvert\rvert_{\infty}$, $\sup\{\lvert s\rvert+\lvert y\rvert:(s,y)\in \text{spt}(\phi)\}$ and $\mu$, which may increase from one line to the next, and define
\[
k_{\delta}(y):=\begin{cases}
    \Pm_y(\text{Brownian motion hits $\{0,1\}$ in time $\delta$}),\quad &\text{if killing boundary at $\{0,1\}$},\\
    0,\quad &\text{if no or reflecting boundary.}
\end{cases}
\]
Regardless of the boundary case, we have $k_\delta(y)=\lvert \int_Ep(\delta,x,y)\rmd x-1\rvert$ by the time-reversibility of Brownian motion, and $\lim_{\delta\ra 0}\int_Ek_{\delta}(y)\rmd y= 0$ by basic properties of Brownian motion.
With these definitions, we have the bounds
\begin{align*}
&\Bigg\lvert\int_0^\infty\int_E\phi(t,x) \frac{1}{h^2}\int_0^h\int_t^{t+h-r}\Big\{\int_{E}D^{\epsilon}(s,y)\rho_{x,\delta}(\rmd y)\Big\}\rmd s\rmd r \rmd x \rmd t-\frac{1}{2}\int_0^{\infty}\int_E \phi(s,y)D^\epsilon(s,y)\rmd s\rmd y\Bigg\rvert \\
&\leq \frac{C}{2}\int_E k_{\delta}(y)\rmd y +\frac{C}{h} \int_E\int_0^{C+h}\int_E\int_{s-h}^s\lvert \phi(t,x)-\phi(s,y)\rvert p(\delta,x,y)\rmd t \rmd x\rmd s\rmd y \\
&\leq \frac{C}{2}\int_E k_{\delta}(y)\rmd y + C (C + h) \sup\{\lvert\phi(t,x)-\phi(s,y)\rvert:\lvert t-s\rvert\leq h,\lvert x-y\rvert\leq \delta\} \\
&\phantom{=} + C^2\int_0^{C+h}\int_E\int_Ep(\delta,x,y)\Ind(\lvert y-x\rvert>\delta)\rmd y\rmd x\rmd s,
\end{align*}
where the last inequality arises by decomposing based on $\lvert x-y\rvert>\delta$ and $\lvert x-y\rvert\leq \delta$.
Using properties of Brownian motion, we get
\begin{align*}
&\limsup_{\delta\ra 0}\limsup_{h\ra 0}\sup_{\epsilon>0}\Bigg\lvert\int_0^\infty\int_E\phi(t,x) \frac{1}{h^2}\int_0^h\int_t^{t+h-r}\Big\{\int_{E}D^{\epsilon}(s,y)\rho_{x,\delta}(\rmd y)\Big\}\rmd s\rmd r \rmd x \rmd t\notag\\
&\phantom{\limsup_{\delta\ra 0}\limsup_{h\ra 0}\sup_{\epsilon>0}\Bigg\lvert}-\frac{1}{2}\int_0^{\infty}\int_E \phi(s,y)D^\epsilon(s,y)\rmd s\rmd y\Bigg\rvert=0.
\end{align*}
We now combine this with \eqref{eq:expression for integral Depsilon number particles} to see that
\[
\limsup_{\delta\ra 0}\limsup_{h\ra 0}\limsup_{\epsilon\ra 0}\Bigg\lvert \int_0^{\infty}\int_E \phi(s,y)D^\epsilon(s,y)\rmd s\rmd y\Bigg\rvert =0.
\]
However, we can remove the $\limsup_{\delta\ra 0}\limsup_{h\ra 0}$ since there is no dependence in the above on them, and we have established vague convergence of $D^\epsilon(s,y)$ to $0$.
\end{proof}

By applying It\^o's lemma to \eqref{eq:martingale for existence using general B} as for \eqref{main_lhs}, with $N=1$ and $v^N=v_0$, we obtain
\begin{align}
&\bbE_{u_0}[\langle H(u_t^{\varepsilon}; x), v_0 \rangle ]- \langle H(u_0^{\varepsilon}; x), v_0 \rangle = \int_0^t\bbE_{u_0}\Big[ A_s^{\varepsilon}(x, v_0) + \mu \hat{B}_s^{\varepsilon}(x, v_0) + \frac{\nu}{2}C_s^{\varepsilon}(x, v_0) \Big] \rmd s,\label{eq:moment for ut from drift hat B}
\end{align}
where $\hat{B}_s^{\epsilon}$ is defined as $B_s^\epsilon$ is defined in \eqref{main_lhs}, with $(b\circ u_s)^{\epsilon}(x^i)/\mu$ (equal to $[(\tilde{b} \circ u_s)^{\varepsilon}(x^i) - u_s^{\varepsilon}(x^i)]$) replaced by $\hat{b}^{\epsilon}(s,x^i)/\mu$. We now take the difference of this with $D^\epsilon_t(x,v_0)$ and apply Proposition \ref{prop:D0=0} to see that
\begin{align}
\int_0^t\expE_{u_0}[(b\circ u_s-\hat{b}_s)^\epsilon(x)]\rmd s=\int_0^t\expE_{u_0}[B_s^\epsilon(x,v_0)-\hat B_s^{\epsilon}(x,v_0)]\rmd s\ra 0
\end{align}
vaguely as $\epsilon\ra 0$ as a function of $(t,x)$), where $\hat{b}_s(x):=\hat{b}(s,x)$.
Thus $\expE_{u_0}[(b\circ u_t-\hat{b}_t)^\epsilon(x)]\ra 0$ vaguely as $\epsilon\ra 0$.
On the other hand, Lebesgue's differentiation theorem and the dominated (bounded) convergence theorem tells us that $\expE_{u_0}[(b\circ u_t-\hat{b}_t)^\epsilon(x)]\ra \expE_{u_0}[(b\circ u_t-\hat{b}_t)(x)]$ for almost every $(t,x)\in \Rm_{>0}\times E$, hence $\expE_{u_0}[(b\circ u_t-\hat{b}_t)(x)]=0$ for almost every $(t,x)\in \Rm_{>0}\times E$. 
Therefore $J_t$ (defined in \eqref{eq:Jt diff mean 0}) has mean $0$.

On the other hand, taking $0$ as the starting time was arbitrary, so equally we have \eqref{eq:conditional expectation 0 for changes in Jt}.
We have now proven Proposition \ref{prop:conditional expectation 0 for changes in Jt}.
\end{proof}

We will now use Proposition \ref{prop:conditional expectation 0 for changes in Jt} to prove \eqref{eq:Jt diff mean 0}. We fix $T>0$ and $\phi\in C_c(E)$. We then define (recall $J_t(x)$ was defined in \eqref{eq:Jt diff mean 0}),
\begin{equation}\label{eq:Mt defin}
M_t:=\int_{E\times E}J_t(x)p(T-t,x,z)\phi(z)\rmd x\rmd z,\quad 0\leq  t< T.
\end{equation}
We claim that $M_t$ is a martingale. Clearly it is uniformly bounded and $(\mathcal{F}_t)_{0\leq t<T}$-adapted. We write $\delta(s,y):=\hat{b}(s,y)-b(u_s(y))$ as a shorthand, and calculate for $0<u<t$:
\begin{align*}
    \expE[M_t\lvert \mathcal{F}_u] &= \expE\Big[\int_{E\times E}\int_0^t\int_Ep(t-s,y,x)\delta(s,y)p(T-t,x,z)\phi(z)\rmd y\rmd s\rmd x\rmd z\Big\lvert \mathcal{F}_u\Big] \\
    &= \int_{E\times E}\underbrace{\expE\Big[\int_u^t\int_Ep(t-s,y,x)\delta(s,y)\rmd y\rmd s\Big\lvert\mathcal{F}_u\Big ]}_{=0\text{ by Proposition \ref{prop:conditional expectation 0 for changes in Jt}}} p(T-t,x,z)\phi(z)\rmd x\rmd z\\
    &\phantom{=} +\int_{E\times E}\int_0^up(t-s,y,x)\delta(s,y)p(T-t,x,z)\phi(z)\rmd y\rmd s\rmd x\rmd z\\
    & =\int_{E\times E}\int_0^u\int_Ep(T-s,y,z)\delta(s,y)\phi(z)\rmd s\rmd x\rmd z\\
    &=\int_{E\times E}\int_0^u\int_Ep(u-s,y,x)p(T-u,x,z)\delta(s,y)\phi(z)\rmd y\rmd s\rmd x\rmd z=M_u.
\end{align*}
Therefore $M_t$ is a martingale.
On the other hand, plugging \eqref{eq:Jt diff mean 0} into \eqref{eq:Mt defin} we have
\begin{align*}
|M_t| &= \Bigg| \int_{E\times E\times E}\int_0^tp(t-s,y,x)p(T-t,x,z)\delta(s,y)\phi(z)\rmd s\rmd x\rmd y\rmd z \Bigg| \\
 & = \Bigg| \int_0^t \int_{E\times E}p(T-s,y,z)\delta(s,y)\phi(z)\rmd y \rmd z \Bigg| \leq \| \delta \|_{\infty} \| \phi \|_{\infty} t.
\end{align*}
using Fubini's theorem and integrating over $x$. It follows that $M_t$ is Lipschitz on $[0,T)$ with $M_0=0$.
Hence $M_t$ is continuous with finite variation, so that $M_t=M_0=0$.
However,
\[
\int_Ep(T-t,x,z)\phi(z)\rmd z\ra \phi(x)
\]
uniformly as $t\uparrow T$, hence 
\[
\lim_{t\uparrow T} M_t = \int_E\phi(x)J_T(x)\rmd x,
\]
almost surely.
Therefore $\int_EJ_T(z)\phi(z)\rmd z=0$ almost surely, for all $\phi\in C_c(E)$, $T>0$. However, $J_t(x)$ is continuous in $(t,x)\in \Rm_{>0}\times E$ by the regularity of the heat kernel, hence $J_t(x)=0$ for all $(t,x)\in \Rm_{>0}\times E$, almost surely.
This tells us that

\begin{equation*}
\int_0^t\int_Ep(t-s,y,x)\hat{b}(s,y)\rmd y\rmd s=\int_0^t\int_Ep(t-s,y,x)b(u_s(y))\rmd y\rmd s,
\end{equation*}
for all $t>0$, $x\in E$, almost surely. Plugging this back into \eqref{eq:martingale for existence using general B}, we are done.\qed

\section{Proof of Theorem \ref{theo:main theorem} for measurable drifts}\label{sec:measurable}

Our goal now is to weaken our assumptions on $b$, to which end assume that $b$ (and hence $\tilde{b}$) are only bounded and measurable.
The additional difficulty compared to before is that we can no longer take a sequence of continuous functions converging pointwise to $b$. 

Suppose that we knew $\Law(u(t,x))\ll \text{Leb}_{(0,1)}+\delta_0+\delta_1$ for all $(t,x)\in \Rm_{>0}\times E$.
By taking a sequence of continuous functions $\tilde{b}^m:[0,1]\ra [0,1]$ converging $(\text{Leb}_{(0,1)}+\delta_0+\delta_1)$-almost everywhere to $\tilde{b}$, we would ensure that $\tilde{b}^m(u(t,x))\ra \tilde{b}(u(t,x))$ almost everywhere, almost surely, and the proof as in the Baire class-1 case would proceed without other changes.
However, we have not been able to prove this.
The trick to circumvent this problem is to take a sequence of continuous functions which depend, in an appropriate way, on the solution to which we apply it.

We begin with a proof of uniqueness, before proving existence in law.
\subsection*{Uniqueness in law}
We take two solutions, $u^1_t$ and $u^2_t$, to \eqref{spde} for a fixed drift $b$ with the same initial condition, $u_0$. We then defined a bounded Borel measure $\nu$ on $[0,1]$ in terms of $u^1$ and $u^2$ as
\[
\nu(\cdot):=\sum_{i=1,2}\sum_{R\in \mathbb{N}}2^{-R}\expE\Bigg[\int_0^R\int_{[-R,R]\cap E}\delta_{u^i(t,x)}(\cdot)\rmd x \rmd t\Bigg].
\]
By Tonelli's theorem,
\begin{equation*}
    \nu(A)>0 \quad \Leftrightarrow \quad \sum_{i=1,2}\Pm(\text{Leb}(\{(t,x):u^i(t,x)\in A\})>0),
\end{equation*}
and by a straightforward application of Lusin's theorem and Tietze's extension theorem, there exists a sequence of continuous functions $\tilde{b}^m:[0,1]\ra [0,1]$ converging $\nu$-almost everywhere to $\tilde{b}$.
As a consequence, for $i=1,2$,
\begin{equation}\label{eq:convergence of drifts a.e. for special nu}
\Pm(\tilde{b}^m(u^i(t,x))\ra \tilde{b}(u^i(t,x)) \text{ as $m\ra\infty$, for Lebesgue a.e.\ $(t,x)\in \Rm_{\geq 0}\times E$)})=1.
\end{equation}

We now take the sequence of duals $(X^{m,1:I_t(m)}_t,V^{m,I_t(m)}_t)$ corresponding to the drifts $\tilde{b}^m$. We will establish that for $N \in \mathbb{N}$, $x^{1:N} \in [0, 1]^N, v^N \in \{0, 1\}^{2^N}$, and any measurable $u_0 : [0, 1] \to [0, 1]$, we have
\begin{align}
&\bbE[ \langle H(u^i_t; x^{1:N}), v^N \rangle ] \notag \\
&= \lim_{m\ra\infty}\bbE[\langle H(u_0; X_t^{m,1:I_t(m)}), V_t^{m,I_t(m)} \rangle | (X_0^{m,1:I_0(m)}, V_0^{m,I_0(m)}) = (x^{1:N}, v^N)]. \label{main_eq limiting second}
\end{align}
This can be done by repeating the proof of Theorem \ref{thm:bernoulli general drift},  replacing the pointwise convergence of $\tilde{b}^m\circ u_s$ to $\tilde{b}\circ u_s$ with \eqref{eq:convergence of drifts a.e. for special nu}, obtaining the convergence of \eqref{eq:sure convergence for extension} almost everywhere and almost surely, and concluding $m \to \infty$ convergence as before. Thus we have \eqref{main_eq limiting second}, and uniqueness in law follows as in the Baire class-1 case.

\subsection*{Existence of weak solutions}

Given bounded and measurable $b:[0,1]\ra \Rm$, we say that $b$ possesses a solution, $u$, if the SPDE \eqref{spde} with drift $b$ possesses a weak solution, $u$. We now prove the following proposition.
\begin{lem}\label{lem:ptwise convergence gives solutions}
Suppose that $(b^n)$ is a sequence of measurable functions, each satisfying $-Cu\leq b^n(u)\leq C(1-u)$ for fixed $C<\infty$, and each possessing a solution $u^n$. Suppose also that $b^n\ra b$ pointwise. Then $b$ possesses a weak solution also.
\end{lem}
\begin{proof}[Proof of Lemma \ref{lem:ptwise convergence gives solutions}]
We follow and adapt the proof in the Baire class-1 case, the complicating factor being that the $u^n$s aren't continuous so don't necessarily, themselves, possess a moment dual. We take $u$ to be the subsequential limit of the $u^n$s, and define
\begin{align}
    \nu := {}& \sum_{n\geq 0}2^{-n}\sum_{R\in \mathbb{N}}2^{-R} \expE\Bigg[\int_0^R\int_{[-R,R]\cap E}\delta_{u^n(t,x)}(\cdot)\rmd x \rmd t\Bigg] \notag\\
    &+ \sum_{R\in \mathbb{N}} 2^{-R} \expE\Bigg[\int_0^R\int_{[-R,R]\cap E}\delta_{u(t,x)}(\cdot)\rmd x\rmd t\Bigg]. \label{eq:nu definition existence}
\end{align}
Using Lusin's theorem and Tietze's extension theorem, we obtain for each $n<\infty$ a sequence of continuous drifts, $(b^{n,m}:m\in \mathbb{N})$, such that (a) $-Cu\leq b^{n,m}\leq C(1-u)$ and $\tilde{b}^{n,m}:=u+b^{n,m}/C$ is continuous and polynomially bounded for each $m<\infty$, and (b) $b^{n,m}\ra b^n$ $\nu$-almost everywhere as $m\ra\infty$.  We let $u^{n,m}$ be the solution for drift $b^{n,m}$, which exists and is unique by the continuity of $b^{n,m}$. It follows from the above convergence (b) and the definition \eqref{eq:nu definition existence} that
\begin{align*}
    &\expE[\text{Leb}(\{(t,x):\tilde{b}^{n,m}(u^n(t,x))\nrightarrow \tilde{b}^n(u^n(t,x))\})]\\
    &\leq \sum_{R\in \mathbb{N}}\expE[\text{Leb}(\{(t,x)\in [0,R]\times ([-R,R]\cap E):\tilde{b}^{n,m}(u^n(t,x))\nrightarrow \tilde{b}^n(u^n(t,x))\})]\\
    &\leq \sum_{R\in\mathbb{N}}2^{R+n}\nu(\{y:\tilde{b}^{n,m}(y)\nrightarrow \tilde{b}^n(y)\quad\text{as $m\ra\infty$}\})=0.
\end{align*}
 Therefore we have
\begin{equation}\label{eq:convergence bnk to bn un a.s.}
\Pm\left(\tilde{b}^{n,m}(u^n(t,x))\ra \tilde{b}^n(u^n(t,x))\quad\text{for Lebesgue almost-every $(t,x)$}\right)=1.
\end{equation}
Then by following the proof in the Baire class-1 case and using \eqref{eq:convergence bnk to bn un a.s.}, it follows that $u^{n,m}\ra u^n$ locally uniformly in distribution as $m\ra\infty$.

Observe that $b^n\ra b$ $\nu$-almost surely.
This is why we need pointwise convergence of $b^n\ra b$ rather than almost-everywhere convergence for some measure: a-priori we don't know anything about $\nu$ so this is the only way to guarantee $\nu$-almost-everywhere convergence of $b^n$ to $b$, since the argument constructs $\nu$ from $(b^n)$ and $b$.

Then by diagonalisation, we obtain a subsequence of drifts, $(b^{n,m_n}:n\in \mathbb{N})$, and corresponding solutions $u^{n,m_n}$ such that the former converges $\nu$-almost everywhere to $b$, and the latter converges locally uniformly in distribution to $u$. Finally, we observe as in \eqref{eq:convergence bnk to bn un a.s.} that 
\begin{equation*}
\Pm\left(\tilde{b}^{n,m_n}(u(t,x))\ra \tilde{b}^n(u(t,x))\quad\text{for Lebesgue almost-every $(t,x)$}\right)=1,
\end{equation*}
and again follow the proof in the Baire class-1 case to see that $u$ is a solution for drift $b$.
\end{proof}
We recall that Borel measurable functions are precisely the union of all Baire class-$\omega$ functions for countable ordinals $\omega$, and this is clearly preserved by restricting to functions bounded between $-Cu$ and $C(1-u)$. Then having established Lemma \ref{lem:ptwise convergence gives solutions}, the existence of a solution for all Borel-measurable $b$ satisfying $-Cu\leq b(u)\leq C(1-u)$ follows by transfinite induction. \qed

\section*{Acknowledgements and data sharing}

The authors are grateful to an anonymous referee for pointing out a gap in an early version of the proof of existence of weak solutions for Baire class-1 drifts.
This work was initiated while Jere Koskela was based at the School of Mathematics, Statistics and Physics, Newcastle University, UK.

No new data were created or analysed in this study.
Data sharing is not applicable to this article.

AI was not used in the creation of this article, beyond basic literature searches.

\bibliography{bibliography}{}
\bibliographystyle{alpha}

\end{document}